\documentclass[%
a4paper,%
]%
{article}

\usepackage{amsmath}
\usepackage{amsthm}
\usepackage{pxfonts}

\usepackage[matrix,arrow]{xy}
\usepackage{graphicx}
\usepackage{hyphenat}
\usepackage{xspace}

\newcommand{\myfig}[1]{\includegraphics{figures/#1}}

\newtheoremstyle{thm}{3pt}{3pt}{\itshape}{}{\bfseries}{}{.5em}{}
\newtheoremstyle{thmsub}{3pt}{3pt}{\upshape}{}{\bfseries}{}{.5em}{}

\theoremstyle{thm}

\newtheorem{theorem}{Theorem}[section]
\newtheorem{lemma}[theorem]{Lemma}
\newtheorem{proposition}[theorem]{Proposition}
\newtheorem{defn}[theorem]{Definition}
\newtheorem{corollary}[theorem]{Corollary}
\newtheorem{thm}{Theorem}
\newtheorem{cor}[thm]{Corollary}

\newcommand{\noproof}{%
    \ifmmode
    \pushQED{\qed}\qedhere
    \else
    {\hspace*{\fill}\qed}%
    \fi}

\numberwithin{equation}{section}

\newtheorem*{qn}{Open Question:}

 \newcommand{\abs}[1]{%
    {\if@display
       \left\lvert #1 \right\rvert
     \else
       \lvert #1 \rvert
     \fi}}
\newcommand{\norm}[1][\cdot]{\left\lVert #1 \right\rVert}

\makeatletter
  \newcommand{\BreakableEnDash}{\leavevmode%
    \prw@zbreak--\discretionary{}{}{}\prw@zbreak}
  \newcommand{\BreakableEmDash}{\leavevmode%
    \prw@zbreak---\discretionary{}{}{}\prw@zbreak}
\makeatother
  \DeclareRobustCommand{\enhyp}{%
    \ifmmode--\else\BreakableEnDash\fi}
  \DeclareRobustCommand{\emhyp}{%
    \ifmmode---\else\BreakableEmDash\fi}

\newcommand{\m}[1]{\ensuremath {\mathcal{#1}}\xspace}
\newcommand{\N}{\ensuremath {\mathbb{N}}\xspace}
\newcommand{\Q}{\ensuremath {\mathbb{Q}}\xspace}
\newcommand{\R}{\ensuremath {\mathbb{R}}\xspace}
\newcommand{\Z}{\ensuremath {\mathbb{Z}}\xspace}

\newcommand{\psb}{\text{psb}}
\newcommand{\ps}{\text{ps}}
\newcommand{\ci}{c^\infty}
\newcommand{\Ci}{\ensuremath {C^\infty}\xspace}
\newcommand{\restrict}{\!\!\mid}
\newcommand{\ssetminus}{\ensuremath {\!\smallsetminus\!}}

\DeclareMathOperator{\map}{Map}
\DeclareMathOperator{\smth}{Diff}
\DeclareMathOperator{\sign}{sign}
\DeclareMathOperator{\supp}{supp}

\begin{document}

\title{%
 The Smooth Structure of the Space of Piecewise\hyp{}Smooth Loops%
}

\date{\today}

\author{Andrew Stacey}

\maketitle

\begin{abstract}%
 We consider the problem of defining the structure of a smooth manifold on the various spaces of piecewise\hyp{}smooth loops in a smooth finite dimensional manifold.
 We succeed for a particular type of piecewise\hyp{}smooth loops.

 We also examine the action of the diffeomorphism group of the circle.
 It is not a useful action on the manifold that we define.
 We consider how one might fix this problem and conclude that it can only be done by completing to the space of loops of bounded variation.
\end{abstract}


\section{Introduction}
\label{sec:intro}

It is often convenient to regard a space of certain loops in a smooth manifold as a smooth manifold itself with the aim of doing differential topology thereon.
Depending on the application this approach can vary from the conceptual to the rigorous.
The two most popular types of loop are continuous and smooth, for both of which there is a rigorous theory of infinite dimensional manifolds making these into smooth manifolds: \cite{pm3,jm,sl,wk,ho}.
Other types of loop have also been considered: it is often convenient to use a manifold modelled on a Hilbert space when one usually uses the space of loops with square\hyp{}integrable first derivative.
Our subject of study is piecewise\hyp{}smooth loops.

These loops are used as a compromise between continuous and smooth loops, having some of the advantages of each over the other.
Like continuous loops, piecewise\hyp{}smooth loops can be pasted together with minimal reparametrisation (none if Moore loops are used).
Like smooth loops, one can parallel transport along piecewise\hyp{}smooth loops.
Also the theory of loop groups, \cite{apgs}, applies to piecewise\hyp{}smooth loops but not to continuous loops.

However when attempting to build a smooth manifold of piecewise\hyp{}smooth loops one encounters the problem that outside the realm of Banach spaces the concept of ``smooth'' becomes increasingly hard to pin down.
If one defines ``smooth'' as ``infinitely differentiable'' then there are many ways to interpret this.
Fortunately a general theory has been developed that is both conceptually simple and straightforward to apply.
This theory has been laid out in the weighty tome \cite{akpm} and it is this calculus that we use for the work in this paper.
The introduction of \cite{akpm} and the historical remarks at the end of the first chapter are an interesting read on the development of calculus in infinite dimensions.

We should clarify at the outset that piecewise\hyp{}smooth loops are often used in a slightly different context in the theory of loop spaces.
It is sometimes the case that a more complicated, or less intuitive, class of loops is required inside which piecewise\hyp{}smooth loops sit as a dense subspace.
One then starts with piecewise\hyp{}smooth loops and completes to the desired space.
Obviously in this case the topology is dictated by the desired completion and not by the space of piecewise\hyp{}smooth loops themselves.
This article has little to say in this context, barring that theorem~\ref{th:open} below can be interpreted as saying that even in this context one should be very precise as to what one means by ``piecewise\hyp{}smooth''.

There is a standard method for making a space of loops in a smooth manifold into a smooth manifold itself.
This is well\hyp{}known for specific examples; the simplest being smooth loops as in \cite[ch IX]{akpm} or \cite{pm3}.
In \cite{math.DG/0612096} we generalised this to an arbitrary class of loops satisfying a short list of conditions.
We shall review these conditions in section~\ref{sec:prelim}.
This reduces the problem of constructing a smooth structure on a space of loops to checking these conditions, all of which only involve loops in \R or \(\R^n\).
The first part of this paper is, therefore, devoted to checking these conditions for the spaces of piecewise\hyp{}smooth loops.

\medskip

The word ``spaces'' in the above is not a misprint.
We say ``spaces'' because we intend to consider two types of piecewise\hyp{}smooth loop: piecewise\hyp{}smooth and piecewise\hyp{}smooth with bounded derivatives; that is, each derivative is a bounded function on its domain of definition.
There is a natural topology on each of these spaces of loops in \R which makes each into a locally convex topological vector space.

\medskip

We begin our analysis with the space of (all) piecewise\hyp{}smooth loops in \R. Our first theorem is perhaps somewhat surprising.

\begin{thm}
 \label{th:open}
 The space of piecewise\hyp{}smooth loops in Euclidean space is a dense topological subspace of the space of continuous loops.
\end{thm}

This means that although the derivatives were used in selecting the loops, when putting a topology on the resulting space this information is thrown away.
It is essentially a consequence of the fact that there are no conditions near breaks; which means that derivatives are ignored near a break.
However, this means that one cannot even test for a break and so effectively one has to ignore all derivatives of all loops.
This introduces all sorts of problems.

This is bad news for building a smooth manifold.
Since not all continuous loops are piecewise\hyp{}smooth, an immediate corollary of this is that the space of piecewise\hyp{}smooth loops is not complete.
One of the foremost conditions that the model space of a smooth (infinite dimensional) manifold must satisfy is a weak form of completeness; referred to as \emph{\(\ci\)\enhyp{}completeness} in \cite{akpm} and more commonly known as \emph{local completeness} in functional analysis, see for example \cite[ch 10]{hj}.
It is generally weaker than sequential completeness but the two are the same for normed vector spaces.
Therefore the space of piecewise\hyp{}smooth loops is not \(\ci\)\enhyp{}complete.
We shall explain later why this is problematic.

\begin{cor}
 The space of piecewise\hyp{}smooth loops in a smooth manifold does not form a smooth manifold in the sense of \cite{akpm}.
\end{cor}

\medskip

When we turn to piecewise\hyp{}smooth and bounded loops then these problems disappear.

\begin{thm}
 \label{th:closed}
 The space of piecewise\hyp{}smooth and bounded loops in a smooth manifold forms a smooth manifold in the sense of \cite{akpm} modelled on the space of piecewise\hyp{}smooth and bounded loops in Euclidean space.
\end{thm}

Unfortunately it is difficult to prove even elementary properties of this manifold:

\begin{qn}
 Does a manifold of piecewise\hyp{}smooth and bounded loops admit smooth partitions of unity?
\end{qn}

We conjecture that the model space is not even smoothly regular, something that we require for a positive answer to the above question.

\medskip

However there is a more serious problem with this manifold.
Consider the natural action of the diffeomorphism group of the circle on this manifold.
It is easy to see that this action is by diffeomorphisms but that is the best that one can say.

\begin{thm}
 \label{th:circle}
 The image of the diffeomorphism group of the circle in the diffeomorphism group of a manifold of piecewise\hyp{}smooth and bounded loops is totally disconnected, even when the latter has the topology of pointwise\hyp{}convergence.
 The image of the circle acting by rigid rotations is discrete.
\end{thm}

The topology of pointwise\hyp{}convergence, also known as the \emph{weak} topology, is the coarsest that one would sanely consider.
Therefore this result holds for any other sensible topology one might try.
The practical upshot of this result is that many standard homotopies which rely on continuously reparametrising a loop do not work directly.
See \cite{math.DG/0612096} for other ways to get these homotopies to work.

\medskip

In the light of this failure we can modify our space slightly to correct our earlier inability to find smooth partitions of unity.
One reason for having piecewise\hyp{}smooth loops is to allow pasting of loops without overmuch reparameterisation.
This process only ever introduces breaks at \emph{rational} points on the circle.
That is to say, if we start with two loops whose breaks are at rational points then the resulting loop will have the same property.
This space is subtly better than that with breaks allowed at \emph{all} points.

\begin{thm}
 \label{th:rational}
 The space of piecewise\hyp{}smooth and bounded loops with breaks only at rational points in a smooth manifold forms a smooth manifold in the sense of \cite{akpm} modelled on the space of such loops in Euclidean space.
 This manifold admits smooth partitions of unity.
\end{thm}

In fact, for pasting loops we need only allow breaks at rational points with denominator a power of \(2\).
This would not introduce any advantages as it is the cardinality of the set of allowable breaks that is important.

We have, of course, lost the action by much of the diffeomorphism group of the circle in that we have to throw out any diffeomorphism which does not map the space of rational points into itself.
In particular, only rational rigid rotations are allowed.
This is not a huge loss as this action was problematic to begin with.

\medskip

In the last part of this paper we return to the space of piecewise\hyp{}smooth and bounded loops with arbitrary breaks.
As the natural circle action is quite an important feature of the structure of a loop space it is worth taking some time to see what happens if we try to fix it.
That is, what does it mean for the topology if we impose the condition that the natural circle action is continuous?
Define a loop \(\gamma\) to be \emph{\(S^1\)\enhyp{}odd} if it satisfies
 \(\gamma(t) + \gamma(\frac12 + t) = 0\)
for all \(t\).

\begin{thm}
 \label{th:other}
 Let \(E\) be a subspace of the space of continuous loops in a Euclidean space with a locally convex vector space topology satisfying the following properties:
 \begin{enumerate}
 \item
   \(E\) contains the space of piecewise\hyp{}smooth and bounded loops and this inclusion is continuous.
 \item
  The given topology on \(E\) is finer than that inherited from the space of continuous loops.
 \item
  The circle action on \(E\) is continuous.
 \item
   \(E\) is complete for its given topology.
 \end{enumerate}
 Then \(E\) contains the subspace of \(S^1\)\enhyp{}odd differentiable loops with derivative of bounded variation and this inclusion is continuous.
\end{thm}

It is easy to adapt this to find other similar subspaces of \(E\).
However, the above is probably already enough to show that the cure is worse than the disease.

\medskip

The conditions that we need to check are given in the prequel to this paper, \cite{math.DG/0612096}.
That paper also contains some useful results which show how topological properties of the model spaces propagate to the manifold.
It also contains a discussion of circle actions on locally convex topological vector spaces that may serve to illustrate just how bad is the circle action on piecewise\hyp{}smooth and bounded loops.

The results in this paper, particularly the negative ones, depend on the analytical properties of the model space.
A certain amount of familiarity with functional analysis is therefore required to follow the arguments.
Our main reference texts are \cite{hs}, \cite{hj}, and \cite{akpm}.
We shall also determine several standard analytical and topological properties of the model spaces beyond those needed to build a manifold partly to illuminate the differences between the various spaces and partly as they can be useful for constructions beyond merely building the manifold.
The most obvious example being the existence of smooth partitions of unity.
Another example is that the construction in the author's paper \cite{math.DG/0505077} uses the fact that the space of smooth loops in Euclidean space is complete, reflexive, and nuclear.
Nuclear spaces are, of course, covered in \cite{hs} and \cite{hj} but the treatise \cite{ap} is also worth a look.

Although this paper is reliant on functional analysis, it is probable that the majority of its readers will be more topologically minded.
We have therefore included section~\ref{sec:pict} to give a little topological insight into the intricacies of the arguments.

\medskip

The rest of the paper is structured as follows.
In section~\ref{sec:prelim} we review the main results of \cite{math.DG/0612096} and list the requirements on the model space for the standard method of making a space of loops into a manifold to work.
In section~\ref{sec:psloops} we shall consider the space of all piecewise\hyp{}smooth loops and prove theorem~\ref{th:open}.
In section~\ref{sec:bound} we shall start our analysis of the space of piecewise\hyp{}smooth and bounded loops by proving theorem~\ref{th:closed}.
We include the proof of theorem~\ref{th:rational} as most of the structure is the same for breaks at only rational points as breaks at arbitrary points.
In section~\ref{sec:diff} we prove theorems~\ref{th:circle} and~\ref{th:other}.

\medskip

We regard the circle as the quotient \(\R/\Z\) and so shall write it additively.
We shall refer to connected subsets of \(S^1\)\emhyp{}including\(S^1\) itself\emhyp{}as \emph{intervals}.
We shall write these intervals in \(S^1\) as if they were intervals in \R without worrying about the wrap\hyp{}around factor.
This will save much annoyance with ``special cases''.
The justification for allowing this abuse is that we shall usually be using this notation when considering issues of continuity and, of course, a map from \(S^1\) is continuous if and only if it is continuous as a map from \R.

\section{An Overview}
\label{sec:pict}

Before we begin the analysis, the following discussion may help the reader understand what is going on.

Let us start by considering our two types of loop.
Consider two loops and suppose that on a small patch one of them looks like the path on the left and the other on the right in figure~\ref{fig:start}.

\begin{figure}
 \begin{center}
  \myfig{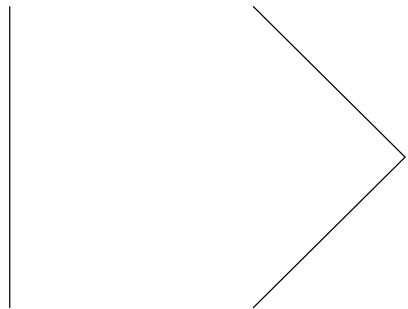}
  \caption{Close\hyp{}up of two paths\label{fig:start}}
 \end{center}
\end{figure}

We wish to interpolate between the two paths.
Firstly, we work in the space of all piecewise\hyp{}smooth loops.
In this space, we never look too closely at what happens near a break.
This results in figure~\ref{fig:unbounded}.
The important thing to notice in this figure is that all of the intermediate paths are smooth; the break develops as the path finally ``snaps''.

\begin{figure}
 \begin{center}
  \myfig{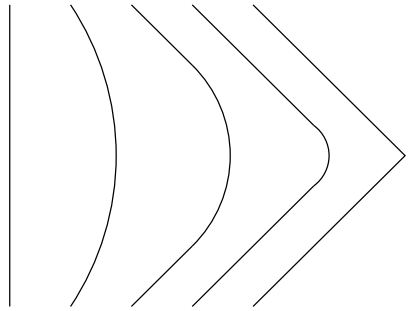}
  \caption{Interpolation I: Piecewise\hyp{}Smooth\label{fig:unbounded}}
 \end{center}
\end{figure}

Secondly, we work in the space of piecewise\hyp{}smooth loops with bounded derivatives.
In this space we are allowed to examine what happens near breaks with the result that breaks cannot just disappear; they have to be gradually phased out.
This results in figure~\ref{fig:bounded}.
Here, all of the intermediate paths have a break which is slowly removed as the two pieces either side of the break come together.

\begin{figure}
 \begin{center}
  \myfig{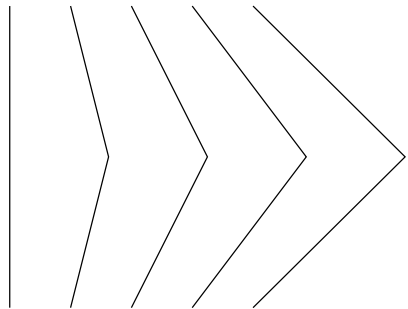}
  \caption{Interpolation II: Piecewise\hyp{}Smooth and Bounded\label{fig:bounded}}
 \end{center}
\end{figure}

Thus piecewise\hyp{}smooth loops are almost smooth in that any piecewise\hyp{}smooth loop can be approximated arbitrarily closely by a smooth loop.
Piecewise\hyp{}smooth bounded loops do not have this property which makes the space of piecewise\hyp{}smooth bounded loops seem larger than that of all piecewise\hyp{}smooth loops.

\medskip

Our second group of pictures is intended to illustrate theorem~\ref{th:open}.
As this is perhaps the most surprising of the results in this paper, it is worth a picture or two to explain it.

Recall that the inductive topology on a union of topological spaces is the finest topology making all the inclusions continuous.
Thus a set in the union is open if all of its traces (intersections) on the pieces are open.
When working with locally convex topological spaces one includes the additional condition that these open sets be locally convex.

With this in mind, we are led to consider subsets of the space of all piecewise\hyp{}smooth loops in a fixed Euclidean space which are  neighbourhoods of the origin whenever we restrict our attention to those loops with a prescribed finite set of breaks.
Theorem~\ref{th:open} states that such a set is a neighbourhood of the origin for the standard topology on the space of continuous loops.
That is to say, for such a set \(U\) there is some \(\epsilon > 0\) such that if \(\abs{\gamma(t)} < \epsilon\) for all \(t \in S^1\) then \(\gamma \in U\).

\begin{figure}
 \begin{center}
  \myfig{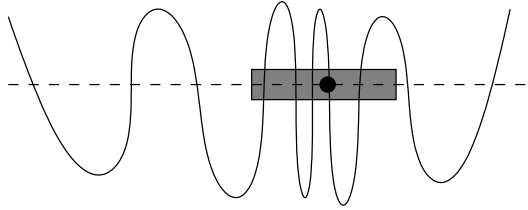}
  \caption{A Loop Close to the Zero Loop\label{fig:open}}
 \end{center}
\end{figure}

The heart of the proof of theorem~\ref{th:open} is illustrated in figure~\ref{fig:open}.
This is a piece of a smooth loop which is close to the zero loop in the \(C^0\)\enhyp{}topology but far away in the \(\Ci\)\enhyp{}topology.
Let us consider it in the piecewise\hyp{}smooth topology.
We restrict our attention at first to those loops which have at most one break and this at the point marked in the figure.
Now as we allow any piecewise\hyp{}smooth loops in our space, we cannot use any information about the derivatives near that point; these are potentially unbounded.
This is true even if our loop did not actually have a break at that point.

This means that there is some neighbourhood of the point, say the grey box in figure~\ref{fig:open}, in which we are only allowed to test the curve itself and not any of its derivatives.
As our original loop was \(C^0\)\enhyp{}close to the zero loop, if we hit it with a suitable bump function with support in the grey box, resulting in figure~\ref{fig:bump}, we are close to zero in the piecewise\hyp{}smooth topology.

\begin{figure}
 \begin{center}
  \myfig{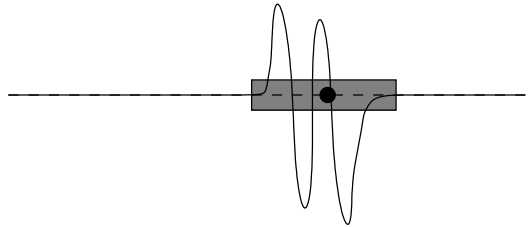}
  \caption{Applying a Bump Function\label{fig:bump}}
 \end{center}
\end{figure}

Therefore our original loop is \emph{locally} close to zero in the piecewise\hyp{}smooth topology.
At this point we recall that we actually wanted the inductive locally convex topology, not just the inductive topology.
This allows us to assume that our open set is convex.
Together with the compactness of \(S^1\), we can now show that
our original loop is \emph{actually} close to zero in the piecewise\hyp{}smooth topology.

\medskip

\begin{figure}
 \begin{center}
  \myfig{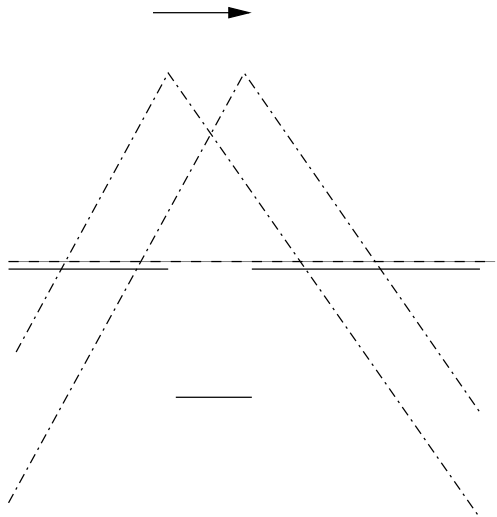}
  \caption{Rotating a Broken Loop\label{fig:rotbreak}}
 \end{center}
\end{figure}

Our last picture illustrates the difficulties with the circle action on piecewise\hyp{}smooth bounded loops.
Consider a loop with a break and rotate it slightly.
Figure~\ref{fig:rotbreak} is a picture of what happens near the break.
The dash\enhyp{}dotted lines are the loop and its rotation.
The solid line is the difference in the derivatives between the two loops.
This is zero except between the breaks where its value is independent of the amount of the rotation.
In the topology on the space of piecewise\hyp{}smooth and bounded loops, this is sufficient to separate the loop from its rotation no matter how small the rotation.
Should one try to fix the problem by taking into account the length of the non\hyp{}zero part then one finds one of two things happening.
Either one tries to apply this fix derivative by derivative in which case the problem is simply shifted up a derivative.
Or if one tries to apply this fix to the whole lot in one go, the net result is that once again derivatives cannot be used in defining the topology.
Broadly speaking, this is theorem~\ref{th:other}.

\medskip

Let us comment on why the negative results are bad news.

Firstly, let us consider theorem~\ref{th:open}.
The corollary of this states that the space of piecewise\hyp{}smooth loops is not complete.
Completeness is required in calculus to avoid issues with non\hyp{}existence of derivatives and integrals for trivial reasons.
This is perhaps easier to understand with integrals.
Using theorem~\ref{th:open} it is simple to construct a continuous curve in the space of piecewise\hyp{}smooth loops which does not have a Riemann integral.
The integral would exist in the space of continuous loops, but not in the space of piecewise\hyp{}smooth loops.
Similarly, it is possible to construct a curve in the space of piecewise\hyp{}smooth loops which ``ought to be'' differentiable but is not; the problem being that the derivative seems to exist but is not a piecewise\hyp{}smooth loop, merely a continuous one.
Thus completeness (or more specifically, \(\ci\)\enhyp{}completeness) is there to ensure that ``things which ought to exist actually do''.
It plays exactly the same r\^ole as completeness does for \R in terms of existence of limits.

The importance of smooth partitions of unity in differential topology is hard to overestimate.
So the possibility that the space of piecewise\hyp{}smooth bounded loops may not have them should at the least make one a little wary of this space.

As mentioned in the introduction, the fact illustrated above that the diffeomorphism group of the circle is totally disconnected as a subgroup of the diffeomorphism group of piecewise\hyp{}smooth bounded loops has a serious impact on many standard homotopies.
Specifically, any homotopy that relies on being able to reparametrise loops cannot be (directly) applied.
Rather one has to appeal to the more general homotopy equivalence between all of these spaces and, say, ordinary smooth loops where the reparameterisation homotopies do work.
Of course, if one has to do that then one may as well work with smooth loops throughout.

\medskip

Finally, let us comment on piecewise\hyp{}smooth loops in the literature.
Although their use is reasonably common, it is rare to find a specific topology mentioned.
A classic example can be found in \cite{kc}.

\begin{quotation}
\textbf{1.3} \emph{Definition.}
A piecewise smooth path (or simply a path) on a differentiable space \(X\) [for example, a finite dimensional smooth manifold] is a continuous map \(\alpha \colon I \to X\) [here, \(I\) is the unit interval] such that, for some partition \(0 = t_0 < \dotsc < t_r = 1\) of the unit interval, each restriction \(\alpha \restrict\!\! [t_{i-1}, t_i]\) is a plot of [smooth map into] \(X\).

Let \(P(X)\) denote the space of all (piecewise smooth) paths on \(X\) with the compact open topology.
\end{quotation}

It is easy to see from this that Chen is using what we call piecewise\hyp{}smooth bounded loops.
However he describes the topology as the ``compact open topology'' and does not elaborate on that.
This is reasonable given that his \emph{differentiable spaces} do not rely overmuch on the underlying topology (indeed, by \cite{kc3} he had dropped the requirement that a differentiable space be a topological space).
However, it does provoke the question as to what is the ``right'' topology on the space of piecewise\hyp{}smooth loops.

It is also interesting to note that by \cite{kc3}, Chen was working with honest smooth loops.
In section~1.5 of \cite{kc3}, a path is an honest plot which means that a loop in a manifold is a smooth loop in the standard sense.
Three paragraphs later, Chen uses the reparameterisation homotopy referred to above which, as we shall see, does not work for piecewise\hyp{}smooth loops.

\section{Manifold Requirements}
\label{sec:prelim}

In this section we shall summarise the work of \cite{math.DG/0612096} and list the conditions on the model space that are required to put a smooth structure on the space of smooth loops in a manifold.

We start by choosing a class of maps \(S^1 \to \R\) which we denote by \(L^x \R\), or by \(C^x(S^1, \R)\) if we wish to emphasise the domain.
We shall refer to these as \(C^x\)\enhyp{}loops.
Already implicit here is the assumption that these are genuine maps so that \(L^x \R\) is a subset of \(\map(S^1,\R)\).
One implication of this is that if we modify \(C^x\) in some fashion, say by completing with respect to some uniformity, then we must be careful to ensure that this completion still consists of genuine maps.
This rules out, for example, \(L^2\)\enhyp{}functions.

From \(L^x \R\) we define some other useful spaces of maps.
We identify the spaces \(\map(S^1, \R^n)\) and \(\map(S^1,\R)^n\) in the obvious way.
With this identification, we define \(L^x \R^n\) as \((L^x \R)^n\).
For \(A \subseteq \R^n\), we define \(L^x A\) as the subset of \(L^x \R^n\) of loops which take values in \(A\).
Here we use the fact that elements in \(L^x \R^n\) are genuine maps and so can be evaluated.

We now list the conditions required.
For why each is important, we refer the reader to \cite{math.DG/0612096}.

\begin{enumerate}
\item
 The condition of being a \(C^x\)\enhyp{}loop is \emph{local}.
 \label{cond:local}

 That is, a loop
  \(\gamma \colon S^1 \to \R\)
 is a \(C^x\)\enhyp{}loop if there is some open cover \(\m{U}\) of \(S^1\) and for each \(U \in \m{U}\) a \(C^x\)\enhyp{}loop \(\gamma_U\) such that \(\gamma\) agrees with \(\gamma_U\) on \(U\).

\item
 The set \(L^x \R\) is a sub\emph{space} of \(\map(S^1,\R)\).
 \label{cond:vspace}

\item
 The vector space \(L^x \R\) can be given a topology with respect to which it is a locally convex topological vector space.
 \label{cond:lctvs}

\item
 With its topology, \(L^x \R\) is a \emph{convenient} vector space.
 That is, it is \(\ci\)\enhyp{}complete or, equivalently, locally complete.
 \label{cond:cmplt}

\item
 As subspaces of \(\map(S^1,\R)\) we have inclusions:
 \[
  L \R \subseteq L^x \R \subseteq L^0 \R
 \]
 where \(L \R = \Ci(S^1, \R)\) and
  \(L^0 \R = C^0(S^1, \R)\).
 These inclusions are continuous with respect to the natural topologies on each.
 \label{cond:smthcts}

\item
 The condition of being \(C^x\) is preserved by post\hyp{}composition by smooth maps and the action of a given smooth map is smooth.
 That is, for \(\phi \colon U \to V\) a smooth map between open sets of Euclidean spaces, the induced map
  \(\phi_* \colon L^x U \to L^x V\)
 is well\hyp{}defined and is smooth in the sense of \cite{akpm}.
 \label{cond:postcomp}
\end{enumerate}

The primary results of \cite{math.DG/0612096} can be summarised in the following theorem.

\begin{theorem}
 Let \(L^x \R\) be a class of loops satisfying the above conditions.
 Let \(M\) be a smooth, finite dimensional, orientable manifold without boundary.
 Then \(L^x M\) can be defined and is a smooth manifold.

 If \(M\) can be embedded as a smooth submanifold of some Euclidean space then the following properties devolve from \(L^x \R\) to \(L^x M\): separable, metrisable, Lindel\"of, paracompact, normal, smoothly regular, smoothly paracompact, and smoothly normal.

 Let \(G\) be a sub\hyp{}Lie group of \(\smth(S^1)\) and suppose that \(L^x \R\) is invariant under the natural action of \(G\).
 Under the same conditions as above, the following properties of this action devolve from \(L^x \R\) to \(L^x M\): that the action is by continuous or smooth maps, that the action map is continuous or smooth, and that the representation map is continuous or smooth.
 \noproof
\end{theorem}

By ``devolve'' we mean that if a property holds for \(L^x \R\) then it holds for \(L^x M\).
The point of this theorem is that it transfers our attention wholly to the space \(L^x \R\) and thus plants us firmly in the realm of functional analysis.
For the action by diffeomorphisms, we shall be most concerned with the case where \(G = S^1\), acting by rigid rotations.

In addition to the properties listed in the above theorem, there are some other properties of \(L^x \R\) that will be useful to know.
These do not appear in the above list because they only make sense for linear spaces and therefore do not have any meaning for \(L^x M\).
However, they will have meaning for the tangent spaces of \(L^x M\) and this may be important.
The example given in the introduction is of the author's construction of the Dirac operator from \cite{math.DG/0505077} which used the fact that the space of smooth loops is a complete, nuclear, reflexive space and therefore we add these to the list of properties that we wish to determine.
For completeness, we have assumed that we have local completeness but this is just a low rung on the panoply of possibilities for completeness; other standard possibilities are complete, quasi\hyp{}complete, and sequentially complete.

In fact, we shall not be able to come up with definite answers for all of the properties that we consider and therefore have to leave some open questions.
We suspect that the answers to these would involve analysis on a far more intricate level that contained in this paper.

For comparison with the spaces of piecewise\hyp{}smooth and piecewise\hyp{}smooth and bounded loops, we list the properties of smooth and continuous loops.

The space of smooth loops, \(L \R\), is: separable, metrisable, Lindel\"of, paracompact, normal, smoothly regular, smoothly paracompact, smoothly normal, complete, nuclear, and reflexive.
The circle action is by diffeomorphisms, the action map is smooth, and the representation map is smooth.

The space of continuous loops, \(L^0 \R\), is: separable, metrisable, Lindel\"of, paracompact, normal, and complete.
It is not smoothly regular, smoothly paracompact, nor smoothly normal; see the remark after the statement of \cite[III.14.9]{akpm} and the references therein.
It is neither nuclear nor reflexive.
The circle action is by homeomorphisms but not diffeomorphisms; the action map is continuous but the representation map is not.

\section{Piecewise\hyp{}Smooth Loops}
\label{sec:psloops}

In this section we turn to the first of our spaces under consideration: piecewise\hyp{}smooth loops.
We start by defining and topologising the set of piecewise\hyp{}smooth maps and then consider the conditions.
We start this by looking at all the conditions other than~\ref{cond:cmplt} and~\ref{cond:postcomp} since these are relatively straightforward.
The main result in this section is that condition~\ref{cond:cmplt} does not hold and so we cannot build a smooth manifold.
We shall not consider condition~\ref{cond:postcomp} directly, although one can deduce from theorem~\ref{th:open} that condition~\ref{cond:postcomp} holds if we put on the space of piecewise\hyp{}smooth loops the induced smooth structure from the space of continuous loops.

\subsection{Definitions}

Before considering the conditions, we must define the space of piecewise\hyp{}smooth loops in \R. It is obvious what this must be: continuous everywhere and smooth except for a finite number of breaks.

\begin{defn}
 A loop
  \(\gamma \colon S^1 \to \R\)
 is said to be \emph{piecewise\hyp{}smooth} if it is continuous and there is some finite set \(F \subseteq S^1\) such that \(\gamma\) is smooth on \(S^1 \ssetminus F\).

 We write the set of all such loops as \(L_{\ps} \R\).
 For a subset \(H \subseteq S^1\) we denote by \(L_{H}\R\) the subset of \(L_{\ps} \R\) consisting of those maps with breaks constrained to lie in \(H\).

 Let \(\m{F}\) denote the set of all finite subsets of \(S^1\).
 For \(H \subseteq S^1\), let \(\m{F}(H)\) denote the set of all finite subsets of \(H\).
\end{defn}

It is clear that \(L_{\ps}\R\) is the union of the sets \(L_{F}\R\) for \(F \in \m{F}\).
The set \(\m{F}\) is directed by inclusion from which we deduce:

\begin{lemma}
  \(L_{\ps} \R\) is a subspace of \(\map(S^1,\R)\); that is, it satisfies condition~\ref{cond:vspace}.
\end{lemma}

\begin{proof}
 Let
  \(\alpha, \beta \in L_{\ps}\R\)
 and \(c \in \R\).
 Let \(A, B \in \m{F}\) be the corresponding subsets of \(S^1\).
 Now \(\alpha\) and \(\beta\) are continuous whence \(\alpha + c \beta\) is also continuous.
 Then \(\alpha\) is smooth on \(S^1 \ssetminus A\), whence also on
  \(S^1 \ssetminus (A \cup B)\),
 and \(\beta\) is smooth on \(S^1 \ssetminus B\), whence also on
  \(S^1 \ssetminus (A \cup B)\).
 Hence \(\alpha + c \beta\) is smooth on
  \(S^1 \ssetminus (A \cup B)\)
 and so
  \(\alpha + c \beta \in L_{\ps} \R\).
\end{proof}

Also in this section we shall verify the locality condition as this does not involve the topology.

\begin{lemma}
 The condition of being piecewise\hyp{}smooth is local; that is, \(L_{\ps} \R\) satisfies condition~\ref{cond:local}.
\end{lemma}

\begin{proof}
 Let
  \(\alpha \in \map(S^1,\R)\)
 be such that there is an open cover \(\m{U}\) of \(S^1\) with functions \(\alpha_U \in L^x \R\) for each \(U \in \m{U}\) such that \(\alpha\) agrees with \(\alpha_U\) on \(U\).
 As continuity is a local property, i.e.~\(L^0 \R\) satisfies condition~\ref{cond:local}, \(\alpha\) is continuous.

As \(S^1\) is compact we can find \(U_1, \dotsc, U_n \in \m{U}\) covering \(S^1\).
Let \(\alpha_j \coloneqq \alpha_{U_j}\).
Let \(F_j\) be the breaks of \(\alpha_j\).
Each \(F_j\) is finite whence the union \(F \coloneqq \bigcup F_j\) is also finite.
As \(U_j \cap F_j \subseteq U_j \cap F\), \(\alpha_j\) is smooth on \(U_j \ssetminus F\).
Hence \(\alpha\) is smooth on each \(U_j \ssetminus F\), whence on \(S^1 \ssetminus F\), and is thus piecewise\hyp{}smooth.
\end{proof}

\subsection{Topology}

Our next task is to topologise \(L_{\ps} \R\).
To do this we use its description as the union of the directed family
 \(\{L_{F} \R, F \in \m{F}\}\).
There is an obvious way to topologise \(L_{F}\R\) for \(F \in \m{F}\).

\begin{defn}
 For \(F \in \m{F}\), define a topology on \(L_{F}\R\) as the projective topology for the maps:
 \begin{align*}
  L_{F}\R &%
  \to L^0\R, \\
  L_{F}\R &%
  \to \Ci(S^1 \ssetminus F, \R).
 \end{align*}
\end{defn}

Here, \(L^0\R\) and
 \(\Ci(S^1 \ssetminus F, \R)\)
are given their standard topologies.
Since \(F\) is a finite subset of \(S^1\), \(S^1 \ssetminus F\) is diffeomorphic to a finite union of open intervals of \R and this identification defines the topology on
 \(\Ci(S^1 \ssetminus F, \R)\).
The projective topology of locally convex topologies is again a locally convex topology, \cite[II.5]{hs}, so this defines a locally convex topological vector space structure on \(L_{F}\R\).
Notice that
 \(L_{\emptyset}\R = L\R\)
and here the given topology agrees with the standard one.

There is a natural way to topologise the union of a directed family: as an inductive limit.
It is important to note that this is the inductive limit in the category of locally convex topological vector spaces and \emph{not} in the category of topological spaces.
The inductive \emph{topology} is not, in general, a locally convex topology.

\begin{defn}
 Define the topology on \(L_{\ps}\R\) as the locally convex inductive limit of the directed family
  \(\{L_{F}\R : F \in \m{F}\}\).
\end{defn}

From the construction of this topology it is easy to verify condition~\ref{cond:smthcts}.

\begin{lemma}
 The inclusions \(L \R \to L_{\ps}\R\) and
  \(L_{\ps} \R \to L^0 \R\)
 are continuous.
\end{lemma}

\begin{proof}
 This follows from the properties of the inductive limit: as \(L\R\) is one of the members in the family its inclusion is continuous; as all the maps \(L_{F}\R \to L^0\R\) are continuous (from the characterisation of the projective topology) the inclusion \(L_{\ps}\R \to L^0\R\) is continuous.
\end{proof}

\subsection{Completeness}

Unfortunately, \(L_{\ps}\R\) is not convenient.
This will come as a corollary of a very surprising result which states that the topology on \(L_{\ps}\R\) is that inherited from \(L^0\R\).
Since \(L_{\ps}\R\) is not closed as a subspace of \(L^0\R\) it cannot therefore be complete.
As \(L_{\ps} \R\) is thus a normed vector space, local completeness (a.k.a.~\(\ci\)\enhyp{}completeness) agrees with ordinary completeness and so \(L_{\ps}\R\) is not convenient.

To prove this result we need to examine the topologies concerned in a little more detail.
We start with our reference spaces.
The space \(L^0\R\) is a Banach space with norm:
\[
 \norm[\gamma]_\infty \coloneqq \sup\{\abs{\gamma(t)} : t \in S^1\}
\]

The space
 \(\Ci(S^1 \ssetminus F,\R)\)
is a Fr\'echet space with semi\hyp{}norms:
\[
 \rho_{C,n}(\gamma) \coloneqq \sup\{\abs{\gamma^{(k)}(t)} : t \in C, 0 \le k \le n\}
\]
for \(n \in \N\) and for \(C\) a  compact subset of \(S^1 \ssetminus F\).
The family of such compact subsets has a countable cofinal (under inclusion) subfamily.
Taking this family yields a countable family of semi\hyp{}norms and hence the structure of a Fr\'echet space.

The topology on \(L_{F}\R\) is the projective topology for its inclusion into the two above spaces.
Since there are only two spaces, the topology of the projective limit is very easy to determine.
Given \(0\)\enhyp{}neighbourhood bases \(\m{U}\) and \(\m{V}\) for the above two spaces, a \(0\)\enhyp{}neighbourhood base for \(L_{F}\R\) is given by the family
 \(\{U \cap V : U \in \m{U}, V \in \m{V}\}\).
By throwing out a few redundancies, we can find a \(0\)\enhyp{}neighbourhood base indexed by \((C,n,\epsilon)\) where \(\epsilon > 0\), \(n \in \N\), and
 \(C \subseteq S^1 \ssetminus F\)
is compact.
The corresponding \(0\)\enhyp{}neighbourhood is:
\begin{gather*}
 U(C,n,\epsilon) \coloneqq \big\{\gamma \in L_{F} \R : \\
  \sup\{ \abs{\gamma^{(k)}(t)} : 1 \le k \le n \text{ and } t \in C \text{ or } k = 0 \text{ and } t \in S^1\} < \epsilon \big\}.
\end{gather*}

The topology on \(L_{\ps}\R\) is the inductive topology of this family, taken over \(F \in \m{F}\).
The method for constructing a \(0\)\enhyp{}neighbourhood base for this is described in \cite[6.6.5]{hj} and \cite[II.6]{hs}.
First choose \(0\)\enhyp{}neighbourhood bases for each of the components, say \(\m{U}_F\).
From each \(\m{U}_F\) we choose one element, \(U_F\), and take the convex hull of their union in \(L_{\ps}\R\).
Doing this for all choices yields a \(0\)\enhyp{}neighbourhood base for \(L_{\ps}\R\).

Thus an element of this \(0\)\enhyp{}neighbourhood base is indexed by a family
 \(\{(C_F,n_F,\epsilon_F) : F \in \m{F}\}\)
where each
 \((C_F, n_F,\epsilon_F)\)
is as above.
The corresponding neighbourhood is:
\begin{align*}
 U((C_F, n_F,\epsilon_F)) \coloneqq \{ \sum_F \lambda_F \gamma_F :\;&
  \lambda_F \in \R, \gamma_F \in L_{F}\R, \\
  & \text{all but finitely many zero,} \\
  & \sum \abs{\lambda_F} \le 1, \\
  &
  \sup\{ \abs{{\gamma_F}^{(k)}(t)} : 1 \le k \le n_F, t \in C_F\} < \epsilon_F, \\
  &
  \sup\{ \abs{\gamma_F(t)} : t \in S^1\} < \epsilon_F\}.
\end{align*}

For a fixed \(F \in \m{F}\) the space \(L_{F}\R\) inherits a topology from its inclusion in \(L_{\ps}\R\).
It is this topology that we wish to examine next.

\begin{proposition}
 \label{prop:inherit}
 Let \(F \in \m{F}\).
 The topology that \(L_{F}\R\) inherits from \(L_{\ps}\R\) agrees with the topology that it inherits from \(L^0\R\).
\end{proposition}

\begin{proof}
 As the inclusion \(L_{\ps}\R \to L^0\R\) is continuous, the inherited topology on \(L_{F}\R\) is at least as fine as that induced by its inclusion in \(L^0\R\).
 Therefore we just need to show that the topology inherited from \(L^0\R\) is at least as fine as that from \(L_{\ps}\R\).

 To do this, we need to show that for each \(0\)\enhyp{}neighbourhood \(U\) on \(L_{\ps}\R\) there is some \(\eta > 0\) such that whenever \(\gamma \in L_{F}\R\) satisfies
  \(\norm[\gamma]_\infty < \eta\)
 then \(\gamma \in U\).
 It is sufficient to do this for a \(0\)\enhyp{}basis and thus for the \(0\)\enhyp{}neighbourhood
  \(U((C_G,n_G,\epsilon_G))\)
 for some fixed but arbitrary index set.

 Thus we fix the family \((C_G,n_G,\epsilon_G)\).
 For \(t \in S^1\), let \(F_t = F \cup \{t\}\) and note that this is in \(\m{F}\).
 Let
  \(C_t \coloneqq C_{F_t}\).
 This is a compact set which does not contain \(t\), hence the family
  \(\{S^1 \ssetminus C_t\}\)
 is an open cover of \(S^1\).
 As \(S^1\) is a compact manifold, there is a finite smooth partition of unity subordinate to this cover, say
  \(\{\tau_j : 1 \le j \le n\}\).
 Let
  \(t_1, \dotsc, t_n \in S^1\)
 be such that the support of \(\tau_j\) is contained in
  \(S^1 \ssetminus C_{t_j}\).
 Let \(F_j = F_{t_j}\), \(C_j = C_{t_j}\),
  \(\epsilon_j = \epsilon_{F_j}\),
 and \(n_j = n_{F_j}\).
 Let
  \(\eta = \frac1n \min\{\epsilon_j\}\)
 and note that this is strictly greater than zero.

 Let \(\gamma \in L_{F}\R\) be such that
  \(\norm[\gamma]_\infty < \eta\).
 For \(1 \le j \le n\) let
  \(\gamma_j = n \tau_j \gamma\).
 It is obvious each \(\gamma_j\) is piecewise\hyp{}smooth with breaks in \(F\) and so is an element of \(L_{F}\R\).
 Then as
  \(L_{F}\R \subseteq L_{F_j}\R\)
 we have
  \(\gamma_j \in L_{F_j}\R\).
 Now
 \[
  \norm[\gamma_j]_\infty = \norm[n \tau_j \gamma]_\infty \le n \norm[\gamma]_\infty < n \eta \le \epsilon_j,
 \]
 whilst \(\gamma_j\) is identically zero on a neighbourhood of \(C_j\) so for \(t \in C_j\) and \(k \in \N\),
  \(\abs{{\gamma_j}^{(k)}(t)} = 0\).
 Hence
  \(\gamma_j \in U(C_j,n_j,\epsilon_j)\)
 and thus
  \(\gamma_j \in U((C_G,n_G,\epsilon_G))\).

 Since
  \(U((C_G,n_G,\epsilon_G))\)
 is a convex set, it contains the following the finite sum:
 \[
  \sum_{j=1}^n \frac1{n} \gamma_j = \sum_{j=1}^n \frac1n n \tau_j \gamma = \gamma.
 \]
 Hence
  \(\gamma \in U((C_G,n_G,\epsilon_G))\).
 This completes the proof.
\end{proof}

It is not quite immediate from this that the topology on \(L_{\ps}\R\) is that inherited from \(L^0\R\) but we are not far off.
There is an important corollary of the above result which we need on our way.

\begin{corollary}
 The subspace \(L\R\) is dense in \(L_{\ps}\R\).
\end{corollary}

\begin{proof}
 For each \(\gamma \in L_{\ps}\R\)
 and each \(0\)\enhyp{}neighbourhood \(U\) then we need to show that there is some \(\beta \in L\R\) such that \(\gamma - \beta \in U\).
 Fix \(\gamma\) and \(U\).
 There is some \(F \in \m{F}\) such that \(\gamma \in L_{F}\R\).
 By proposition~\ref{prop:inherit}, the inherited topology on \(L_{F}\R\) agrees with that from \(L^0\R\).
 Thus there is some \(0\)\enhyp{}neighbourhood \(V\) in \(L^0\R\) such that
  \(V \cap L_{F}\R \subseteq U \cap L_{F}\R\).
 Now \(L\R\) is dense in \(L^0\R\) so there is some \(\beta \in L\R\) such that \(\gamma - \beta \in V\).
 Since \(L\R\) is a subspace of \(L_{F}\R\),
  \(\gamma - \beta \in L_{F}\R\)
 also.
 Hence \(\gamma - \beta \in U\).
\end{proof}

\begin{corollary}
 The topology on \(L_{\ps}\R\) is that inherited from \(L^0\R\).
\end{corollary}

\begin{proof}
 It is a corollary of the existence and uniqueness of completions of locally convex topological vector spaces that the topology on a locally convex topological vector space is completely determined by its trace on a dense subspace.
 Thus since \(L\R\) is dense in \(L_{\ps}\R\) and the trace topology agrees with that inherited from \(L^0\R\), this must also be the topology on \(L_{\ps}\R\).
\end{proof}

We can rephrase this result using the universal nature of inductive topologies.

\begin{corollary}
 Let \(\m{T}\) be a locally convex topology on \(L_{\ps}\R\) such that all the maps
  \(L_{F}\R \to L_{\ps}\R\)
 and the map \(L_{\ps}\R \to L^0\R\) are continuous.
 Then \(\m{T}\) agrees with the topology inherited from \(L^0\R\).
 \noproof
\end{corollary}

Finally, we state the failure of \(L_{\ps}\R\) to be a good model space for a manifold, and also of any of the spaces derived from it.

\begin{corollary}
  \(L_{\ps}\R\) does not satisfy condition~\ref{cond:cmplt}.
\end{corollary}

\begin{proof}
 The space \(L_{\ps}\R\) is a \emph{topological} subspace of \(L^0\R\) so is a normed vector space.
 As such it is convenient if and only if it is complete.
 As it is dense in \(L^0\R\) but is not equal to \(L^0\R\), it cannot be complete.
\end{proof}

\section{Piecewise\hyp{}Smooth and Bounded}
\label{sec:bound}

Having fallen at the first significant hurdle with piecewise\hyp{}smooth maps we now consider an alternative.
The problem with piecewise\hyp{}smooth maps is that we have to deny ourselves any control over the maps in the neighbourhood of a break.
In fact, we have to even deny ourselves the knowledge that there is a genuine break at a given point, and this\emhyp{}we have seen\emhyp{}leads to all sorts of trouble.
This suggests refining our type of maps so that we can impose some sort of order at the (potential) breaks.

\begin{defn}
 A \emph{piecewise\hyp{}smooth bounded} map
  \(\gamma \colon S^1 \to \R\)
 is a piecewise\hyp{}smooth map with the property that each derivative is bounded on its domain of definition.

 We write the set of all such maps as \(L_{\psb}\R\).
 For \(H \subseteq S^1\) we denote by \(L_{H,b}\R\) the set of all piecewise\hyp{}smooth bounded maps whose breaks lie in the set \(H\).
\end{defn}

In the piecewise\hyp{}smooth world, \(L_{H} \R\) for \(H\) other than \(S^1\) was merely a way\hyp{}stone on the path to \(L_{\ps}\R\).
In this case, we shall actually want to consider \(L_{H,b}\R\) as an end in itself.

Simple calculus yields the following result.

\begin{lemma}
 \label{lem:limits}
 Let
  \(\gamma \colon S^1 \to \R\)
 be a piecewise\hyp{}smooth map.
 Then \(\gamma\) is a piecewise\hyp{}smooth bounded map if and only if all the derivatives of \(\gamma\) have left and right limits at all points of \R.
 \noproof
\end{lemma}

This lemma allows us to ignore the breaks of a loop when dealing with its boundedness properties.
We can regard the derivatives as being effectively defined (albeit possibly multivalued) at the breaks and thus on the whole of \(S^1\).
As these extensions are still bounded we can allow ourselves the freedom to ignore issues of domains of definition when considering bounds on a given loop or loops.

This lemma provides us with straight\hyp{}forward verification of the first two conditions.

\begin{corollary}
  \(L_{H,b}\R\) is a subspace of \(L_{\ps}\R\) and thus satisfies condition~\ref{cond:vspace}.
\end{corollary}

\begin{proof}
 Let
  \(\alpha,\beta \in L_{H,b}\R\)
 and \(\lambda \in \R\).
 We know already that
  \(\alpha + \lambda \beta\)
 is piecewise\hyp{}smooth and has breaks in \(H\) so we just need to show the boundedness property.
 This follows from basic results on limits using lemma~\ref{lem:limits}.
\end{proof}

\begin{corollary}
  \(L_{H,b}\R\) satisfies the locality condition,~\ref{cond:local}.
\end{corollary}

\begin{proof}
 Let
  \(\gamma \colon S^1 \to \R\)
 be a map which is locally a piecewise\hyp{}smooth bounded map.
 As the set of piecewise\hyp{}smooth maps has the local property, \(\gamma\) is piecewise\hyp{}smooth.
 Thus we need to show that it has the boundedness property.
 This follows from lemma~\ref{lem:limits}: for \(t \in S^1\) there is a neighbourhood \(I\) and a loop \(\alpha\) such that \(\gamma = \alpha\) on \(I\) and \(\alpha\) has left and right limits of all derivatives at \(t\), whence so does \(\gamma\).
\end{proof}

\subsection{Topology}

Now we come to the topology.
We use essentially the same method as in the piecewise\hyp{}smooth case: first topologise the spaces \(L_{F,b}\R\) for \(F \in \m{F}\) and then express \(L_{H,b}\R\) as the inductive limit of these spaces for \(F \in \m{F}(H)\).

The most obvious way to describe the topology on \(L_{F,b}\R\) is as the topology of uniform convergence of derivatives on \(S^1 \ssetminus F\).
Using lemma~\ref{lem:limits} and a standard \(\epsilon/2\)\enhyp{}argument we can replace this with the topology of uniform convergence of left and right limits of derivatives on the whole of \(S^1\).
That is, a \(0\)\enhyp{}neighbourhood base consists of the sets:
\[
 U(n,\epsilon) \coloneqq \{\gamma \in L_{F,b}\R : \sup\{\abs{\gamma_\pm^{(k)}(t)} : t \in S^1, 0 \le k \le n\} < \epsilon\}.
\]

\begin{defn}
 Define the topology on \(L_{H,b}\R\) to be the inductive locally convex topology from the family
  \(\{L_{F,b}\R, F \in \m{F}(H)\}\).
\end{defn}

We now turn to checking the simpler conditions.

\begin{lemma}
  \(L_{H,b}\R\) sits nicely between smooth and continuous maps, that is it satisfies condition~\ref{cond:smthcts}.
\end{lemma}

\begin{proof}
 That smooth loops are piecewise\hyp{}smooth bounded loops is obvious as the circle is compact.
 Thus
  \(L_{\emptyset,b}\R = L\R\);
 note that
  \(\emptyset \in \m{F}(H)\)
 whatever \(H\) is.
 It is then trivial to note that the topology on \(L \R\) is the same as that on \(L_{\emptyset,b} \R\).

 The inclusion \(L_{H,b}\R \to L^0\R\) is immediate from the definition.
 To show that it is continuous, we just need to show that the inclusions \(L_{F,b}\R \to L^0\R\) are continuous; the desired result then follows from the universal property of inductive limits.
 This follows from the simple observation that \(U(0,1)\) is the intersection of \(L_{F,b}\R\) with the unit ball in \(L^0\R\).
\end{proof}

\subsection{Completeness}

Having checked the simple conditions, we turn to completeness.
Most of the properties of \(L_{H,b}\R\) that are in the realm of functional analysis stem from the following two technical results.

\begin{proposition}
 \label{prop:subtop}
 Let \(F \in \m{F}(H)\).
 The topology on \(L_{F,b}\R\) is the same as that inherited from its inclusion in \(L_{H,b}\R\).
\end{proposition}

\begin{proof}
 The inclusion
  \(L_{F,b}\R \to L_{H,b}\R\)
 is continuous by construction.
 Therefore we need to show that if \(U\) is a \(0\)\enhyp{}neighbourhood in \(L_{F,b}\R\) then there is a \(0\)\enhyp{}neighbourhood \(V\) in \(L_{H,b}\R\) such that
  \(V \cap L_{F,b}\R \subseteq U\).

 Now a \(0\)\enhyp{}neighbourhood base of \(L_{F,b}\R\) is given by the family of sets \(U(n,\epsilon)\).
 A \(0\)\enhyp{}neighbourhood base of \(L_{H,b}\R\) is given by the family of sets \(U((n_G,\epsilon_G))\) where \(G\) runs over \(\m{F}(H)\) and for each \(G\), \(n_G \in \N\) and \(\epsilon_G > 0\).
 Then:
 \begin{align*}
  U((n_G,\epsilon_G)) \coloneqq \big\{ \sum_G \lambda_G \gamma_G &%
   : \lambda_G \in \R, \gamma_G \in L_{G,b}\R, \\
   & \text{ all but finitely many zero,} \\
   & \sum \abs{\lambda_G} \le 1, \\
   &
   \sup \{\abs{\gamma^{(k)}_{G,\pm}(t)} : t \in S^1, 0 \le k \le n_G\} < \epsilon_G\big\}.
 \end{align*}

 Let \(\epsilon > 0\) and \(n \in \N\).
 We need to find a set of the above type such that if
  \(\gamma \in U((n_G, \epsilon_G))\)
 and \(\gamma \in L_{F,b}\R\)
 then
  \(\gamma \in U(n,\epsilon)\).
 For \(G \in \m{F}(H)\), let \(n_G = n\) and
  \(\epsilon_G = \epsilon\).
 Let \(\gamma \in L_{F,b}\R\)
 be such that
  \(\gamma \in U((n_G,\epsilon_G))\).
 Then
  \(\gamma = \sum_G \lambda_G \gamma_G\)
 for appropriate \(\lambda_G\) and \(\gamma_G\).

 Let \(t \in S^1\) and \(k \in \N\) with \(0 \le k \le n\).
 Now it may be the case that \(\gamma^{(k)}(t)\) and some or all of the \(\gamma_G^{(k)}(t)\) are not defined.
 However, the left and right limits of all are defined and satisfy:
 \[
  \gamma^{(k)}_+(t) = \sum_G \lambda_G \gamma^{(k)}_{G,+}(t)
 \]
 and similarly for the left limit.
 For \(G \in \m{F}(H)\), as
  \(\gamma_G \in U(n_G,\epsilon_G)\),
  \(\abs{\gamma_{G,+}^{(k)}(t)} < \epsilon\).
 Thus as
  \(\sum \abs{\lambda_G} \le 1\),
  \(\abs{\gamma^{(k)}(t)} < \epsilon\).
 As this holds for all \(t \in S^1\),
  \(\gamma \in U(n, \epsilon)\)
 as required.
\end{proof}

\begin{proposition}
 \label{prop:bounded}
 Let
  \(K \subseteq L_{H,b}\R\)
 be a bounded set.
 Then there is some \(F \in \m{F}(H)\) such that
  \(K \subseteq L_{F,b}\R\).
\end{proposition}

\begin{proof}
 We need to show that there is some \(F \in \m{F}(H)\) such that if \(\gamma \in K\) then the breaks of \(\gamma\) lie in the set \(F\).
 Let \(F_K \subseteq S^1\) be the set of all breaks of elements of \(K\); that is,
  \(F_K = \bigcup_{\gamma \in K} F(\gamma)\)
 where \(F(\gamma)\) is the set of breaks of \(\gamma\).
 Clearly \(F_K \subseteq H\) and
  \(K \subseteq L_{F_K,b}\R\).
 Therefore we just need to show that \(F_K\) is finite.

 For \(t \in S^1\), define a map
  \(\lambda_t \colon L_{H,b}\R \to \R^\N\)
 by:
 \[
  \lambda_t(\gamma) = (\gamma'_+(t) - \gamma'_-(t), \dotsc, \gamma^{(k)}_+(t) - \gamma^{(k)}_-(t), \dotsc).
 \]
 A simple corollary of Borel's theorem, see \cite{jw} or \cite[15.4,21.5]{akpm}, shows that \(\lambda_t\) is surjective.
 The topology on \(\R^\N\) has \(0\)\enhyp{}basis the sets \(V(n,\epsilon)\) with \(n \in \N\), \(\epsilon > 0\):
 \[
  V(n,\epsilon) \coloneqq \{(a_k) : \abs{a_l} < \epsilon, 1 \le l \le n\}.
 \]
 For \(F \in \m{F}(H)\) let \(n_F = n\) and let
  \(\epsilon_F = \epsilon/2\).
 Let
  \(\gamma \in U((n_F,\epsilon_F))\).
 By standard arguments, for \(0 \le k \le n\),
  \(\abs{\gamma^{(k)}_\pm(t)} < \epsilon/2\).
 Hence
  \(\abs{\gamma^{(k)}_+(t) - \gamma^{(k)}_-(t)} < \epsilon\)
 and so
  \(\lambda_t(\gamma) \in V(n,\epsilon)\).
 Thus \(\lambda_t\) is continuous.

 Consider the space \(\sum_{t \in H} \R^\N\).
 Define a map
  \(\lambda_H \colon L_{H,b}\R \to \sum_{t \in H} \R^\N\)
 by
  \(\gamma \mapsto \sum_{t \in S^1} \lambda_t(\gamma)\).
 Now
  \(\lambda_t(\gamma) \ne 0\)
 only if \(\gamma\) has a break at \(t\) so for a given \(\gamma\), \(\lambda_H(\gamma)\) has only a finite number of non\hyp{}zero terms; hence \(\lambda_{H}\) is well\hyp{}defined.
 For a fixed \(F \in \m{F}(H)\), the induced map
  \(L_{F,b}\R \to \sum_{t \in H} \R^\N\)
 fits into the diagram:
 \[
  \xymatrix{ L_{H,b}\R \ar[r]^{\lambda_{H}} &
   \sum_{t \in H} \R^\N \\
   L_{F,b}\R \ar[r]^{\lambda_{F}} \ar[u] &
   \sum_{t \in F} \R^\N \ar[u] }
 \]
 The lower horizontal map is a finite product of \(\lambda_t\), hence is continuous.
 Therefore \(\lambda_{H}\) restricts to a continuous map
  \(L_{F,b}\R \to \sum_{t \in H} \R^\N\).
 Hence by the universal property of inductive limits, \(\lambda_{H}\) is continuous.

 Now a continuous linear map takes bounded sets to bounded sets.
 Therefore \(\lambda_{H}(K)\) is bounded.
 From \cite[II.6.3]{hs}, a bounded set in a direct sum is contained in a finite number of its factors.
 Therefore the set
  \(\{t \in H : \lambda_t(\gamma) \ne 0 \text{ for some } \gamma \in K\}\)
 is finite.
 This is precisely \(F_K\).
\end{proof}

\begin{corollary}
  \(L_{H,b}\R\) is convenient; that is, it satisfies condition~\ref{cond:cmplt}.
\end{corollary}

\begin{proof}
 We shall actually show that it is quasi\hyp{}complete, that is that all closed, bounded subsets are complete.
 This is stronger than \(\ci\)\enhyp{}completeness.

 Let \(K\) be a closed, bounded set in \(L_{H,b}\R\).
 By proposition~\ref{prop:bounded} there is some \(F \in \m{F}(H)\) such that
  \(K \subseteq L_{F,b}\R\).
 Proposition~\ref{prop:subtop} shows that the induced topology on \(L_{F,b}\R\) is the natural one, whence \(K\) is a closed and bounded subset of \(L_{F,b}\R\).

 Now \(L_{F,b}\R\) is a Fr\'echet space and thus quasi\hyp{}complete.
 Thus \(K\) is complete.
\end{proof}

We postpone the more general question as to whether or not \(L_{H,b}\R\) is complete to section~\ref{sec:furprop}.

\subsection{Smooth Maps}

A second corollary of proposition~\ref{prop:bounded} is important in examining the smooth structure.

\begin{corollary}
 Let
  \(c \colon \R \to L_{H,b}\R\)
 be a continuous curve.
 Then for each \(r > 0\) there is some \(F \in \m{F}(H)\) such that
  \(c([-r,r]) \subseteq L_{F,b}\R\).
\end{corollary}

\begin{proof}
 As \([-r,r]\) is compact and \(c\) is continuous, its image is bounded and hence contained in some \(L_{F,b}\R\).
\end{proof}

This will allow us to reduce the problem of checking condition~\ref{cond:postcomp} to \(L_{F,b}\R\).

\begin{lemma}
 Let \(U \subseteq \R^n\) and \(V \subseteq \R^m\) be open sets.
 Let \(\phi \colon U \to V\) be a smooth map.
 Then the map
  \(\phi_* \colon \gamma \mapsto \phi \circ \gamma\)
 is a smooth map
  \(L_{F,b} U \to L_{F,b}V\).
\end{lemma}

\begin{proof}
 Firstly, it is obvious that \(\phi \circ \gamma\) does indeed lie in \(L_{F,b}V\).
 Thus the map is defined and we simply need to check that it is smooth.
 We are using the convenient calculus of \cite{akpm} so we need to check that \(\phi_*\) takes smooth curves to smooth curves.
 That is, we need to show that if
  \(c \colon \R \to L_{F,b}U\)
 is smooth then so is
  \(\phi_* c \colon \R \to L_{F,b}V\).
 Smoothness is a local concept in all its forms so it is sufficient to show that if
  \(c \colon \R \to L_{F,b}U\)
 is smooth then for every \emph{bounded} open interval \(I \subseteq \R\) then the restriction of \(\phi_* c\) to \(I\) is smooth.
 The reason for doing this is that it allows us to modify \(\phi\) slightly.
 Choose a sequence \(\{U_n\}\) of open subsets of \(U\) with the property that
  \(\overline{U_n} \subseteq U_{n+1}\)
 and \(U = \bigcup U_n\); as the circle is compact we see that
  \(L_{F,b}U = \bigcup L_{F,b} U_n\).
 Then as \(I \subseteq \R\) has compact closure, \(c\) maps \(I\) into \(L_{F,b}U_n\) for some \(n\).
 Applying the locality of smoothness again, it is therefore sufficient to check that
  \(\phi_* \colon L_{F,b} U_n \to L_{F,b}V\)
 is smooth for each \(n\).
 Now on the right we have that \(L_{F,b} V\) is an open subset of \(L_{F,b} \R^n\) so a map into \(L_{F,b} V\) is smooth if and only if it is smooth into \(L_{F,b} \R^n\).
 On the left, since
  \(\overline{U_n} \subseteq U\)
 we can use a bump function to find a map
  \(\tilde{\phi} \colon \R^m \to \R^n\)
 which agrees with \(\phi\) on \(U_n\).
 On \(I\),
  \(\phi_* c = \tilde{\phi}_* c\)
 and therefore it is sufficient to show that \(\tilde{\phi}_* c\) is a smooth map \(I \to L_{F,b} \R^n\).

 To do this we need to characterise smooth curves in \(L_{F,b}\R^n\).
 Now \(S^1 \ssetminus F\) is diffeomorphic to a disjoint union of unit intervals.
 The obvious map:
 \[
  L_{F,b} \R^n \to \prod_{i=1}^{k} \Ci_b((0,1)), \R^n)
 \]
 is injective and a homeomorphism onto its image.
 The image is:
 \[
  \{(f_1, \dotsc, f_k) : f_{i,-}(1) = f_{i+1,+}(0), f_{k,-}(1) = f_{1,+}(0)\},
 \]
 which has finite codimension and is thus a direct summand.
 Now since the left and right limits exist,
  \(\Ci_b((0,1),\R^n) = \Ci([0,1],\R^n)\).
 By Seeley's theorem, \cite{rs}, this is a direct summand of \(\Ci(\R,\R^n)\).
 Therefore \(L_{F,b} \R^n\) is a direct summand of
  \(\prod_{i=0}^{n-1} \Ci(\R, \R^n)\).

 Following all of this through, a curve in \(L_{F,b} \R^n\) is smooth if and only if it is smooth into
  \(\prod_{i=0}^{n-1} \Ci(\R,\R^n)\)
 and a map into a finite product is smooth if and only if it is smooth into each factor.
 The exponential law of \cite[3.2]{akpm} says that a curve in \(\Ci(\R,\R^n)\) is smooth if and only if its adjoint, which is a map \(\R^2 \to \R^n\), is smooth.

 All of this jiggery\hyp{}pokery has been to do with the domains.
 We have not touched the codomains.
 It is easy to see, therefore, that if
  \(c \colon \R \to L_{F,b}\R^m\)
 is and
  \(c^\lor \colon \R^2 \to \R^n\)
 the result of the above mechanics then
  \((\tilde{\phi}_* c)^\lor = \tilde{\phi} \circ c^\lor\).
 Hence \(\tilde{\phi}_*\) takes smooth curves to smooth curves and is thus smooth.
\end{proof}

\begin{corollary}
 The map
  \(\phi_* \colon L_{H,b}U \to L_{H,b}V\)
 is smooth.
 Hence \(L_{H,b} \R\) satisfies condition~\ref{cond:postcomp}.
\end{corollary}

\begin{proof}
 We need to show that if
  \(c \colon \R \to L_{H,b}U\)
 is smooth then
  \(\phi_* c \colon \R \to L_{H,b}V\)
 is smooth.
 As smoothness is a local concept, this holds if it is true for \(c\) restricted to \((-r,r)\) for all \(r > 0\).
 Now for \(r > 0\), let \(c_r\) be this restriction, then \(c_r\) is a smooth curve in \(L_{F,b} U\) for some \(F \in \m{F}(H)\).
 By the above, \(\phi_* c_r\) is a smooth curve in \(L_{F,b}V\), whence also smooth in \(L_{H,b}V\).
 Hence \(\phi_*\) is smooth.
\end{proof}

\subsection{Further Properties}
\label{sec:furprop}

We now consider some of the other properties for \(L_{H,b} \R\).
For this consideration we need to consider two cases: where \(H\) is countable and where it is not.
Our primary examples are \(H = \Q / \Z\) and \(H = S^1\).

In the following theorem, each space under consideration has two interesting topologies: its original locally convex topology and its \(\ci\)\enhyp{}topology.
Some of the properties that we consider are properties of locally convex topological vector spaces and these clearly only make sense for the former topology.
Others are clearly in the realm of the smooth structure.
Whilst they may make sense for both topologies they are most interesting for the \(\ci\)\enhyp{}topology.
Still other properties are relevant for both topologies.
Some standard topological properties, in particular separation properties, can be given a smooth twist.
For example, one can alter the usual topological property of complete regularity to that of smooth regularity where the separating function is required to be smooth.
In all the cases that we consider, the smooth version is an obvious alteration to the standard one so we shall not list them all.

\begin{theorem}
 Let \(H \subseteq S^1\).
 \begin{enumerate}
 \item
 Suppose that \(H\) is countably infinite.
 \begin{enumerate}
  \item As a locally convex topological vector space, \(L_{H,b} \R\) is: complete, nuclear, reflexive, barrelled, and bornological.
Its topology is normal, separable, Lindel\"of, and paracompact, but not metrisable.

  \item The \(\ci\)\enhyp{}topology on \(L_{H,b} \R\) is the inductive topology \emph{as a topological space} of the family
   \(\{L_{F,b}\R : F \in \m{F}(H)\}\).
  It is separable, Lindel\"of, smoothly Hausdorff, smoothly paracompact, and smoothly normal, but not metrisable.
 \end{enumerate}

 \item
  Suppose that \(H\) is uncountable.
  \begin{enumerate}
   \item As a locally convex topological vector space, \(L_{H,b} \R\) is: complete, reflexive, barrelled, bornological, but not nuclear.
Its topology is not separable, metrisable, or Lindel\"of.
We do not know whether or not it is paracompact or normal.

  \item The \(\ci\)\enhyp{}topology on \(L_{H,b} \R\) is the inductive topology \emph{as a topological space} of the family
   \(\{L_{F,b} \R : F \in \m{F}(H)\}\).
  It is smoothly Hausdorff.
  It is not separable, Lindel\"of, or metrisable.
  We do not know whether or not it is regular, paracompact, normal, smoothly regular, smoothly paracompact, or smoothly normal.
  \noproof
  \end{enumerate}
 \end{enumerate}
\end{theorem}

For \(H\) uncountable we leave open several questions.
The techniques we use to answer these questions for \(H\) countable do not extend to the uncountable case.
In light of the fact that, as we shall see, there is not a significant advantage to taking \(H = S^1\) over taking \(H = \Q / \Z\), we leave these questions to the future, though we shall make some remarks on them at the end of this section.

\subsubsection{The Locally Convex Topology}

Let us start by considering the locally convex topology on \(L_{H,b} \R\).

\begin{proposition}
 For \(H \subseteq S^1\), \(L_{H,b}\R\) is reflexive, barrelled, and bornological.
\end{proposition}

\begin{proof}
 Both barrelled and bornological are preserved by inductive limits and  each \(L_{F,b} \R\) is both barrelled and bornological; see  \cite[II.7,II.8]{hs}.
 For reflexivity, we use proposition~\ref{prop:bounded}.
 Each \(L_{F,b} \R\) is a nuclear Fr\'echet space, whence reflexive,  so using proposition~\ref{prop:bounded} we can apply  \cite[11.4.5(e)]{hj} to deduce that \(L_{H,b}\R\) is also reflexive.
\end{proof}

For \(H\) countable much of the rest of our analysis relies on the following result.

\begin{proposition}
 Let \(H \subseteq S^1\) be countable.
 Then \(L_{H,b}\R\) is the strict inductive limit of a sequence \(\{L_{F_n,b}\R\}\).
\end{proposition}

\begin{proof}
 Proposition~\ref{prop:subtop} shows that if \(F \subseteq G\) then the topology on \(L_{F,b}\R\) is the same as that inherited from \(L_{G,b}\R\).
 This is what is meant by the word ``strict'' in the definition of an inductive limit.
 In any inductive limit, we can replace the family by a cofinal subfamily; therefore to complete this proof we need to exhibit an increasing sequence in \(\m{F}(H)\) which is cofinal.
 Enumerate \(H\) and let
  \(F_n \coloneqq \{h_1, \dotsc, h_n\}\).
 It is easy to see that this is increasing and every finite subset of \(H\) is contained in one of its terms.
 Hence this will do for the sequence.
\end{proof}

\begin{corollary}
 If \(H \subseteq S^1\) is countable, then \(L_{H,b}\R\) is complete, nuclear, separable, Lindel\"of, paracompact, and normal.
\end{corollary}

\begin{proof}
 It is complete by \cite[II.6.6]{hs} and nuclear by \cite[III.7.4]{hs}.

 Each \(L_{F,b}\R\) for \(F \in \m{F}(H)\) is separable so as the countable union of separable subspaces, \(L_{H,b}\R\) is separable.

 For Lindel\"of, let \(\m{U}\) be an open cover of \(L_{H,b}\R\).
 Each \(L_{F,b}\R\) is a separable metrisable space, hence Lindel\"of, with its inherited topology.
 Let \((F_n)\) be a cofinal sequence of elements of \(\m{F}(H)\).
 For each \(n \in \N\) there is therefore a countable subfamily of \(\m{U}\) which covers \(L_{F_n,b} \R\).
 As \(L_{H,b}\R\) is the union of the countable family \(\{L_{F_n,b}\R\}\), the union of these countable subfamilies is a covering family and is also countable.

 Paracompactness now follows as every regular Lindel\"of space is paracompact, as does normality.
\end{proof}

The other properties that are firmly in the realm of functional analysis come from a closer examination of the maps:
\[
 \lambda_t \colon L_{H,b}\R \to \R^\N
\]
that were defined in proposition~\ref{prop:bounded}.

\begin{proposition}
 As in the proof of proposition~\ref{prop:bounded}, for \(H \subseteq S^1\) let
  \(\lambda_H \colon L_{H,b}\R \to \sum_{t \in H} \R^\N\)
 be the map
  \(\sum_{t \in H} \lambda_t\).
 This is a quotient map with kernel \(L\R\).
\end{proposition}

\begin{proof}
 It is clearly well\hyp{}defined.
 Borel's theorem together with the existence of smooth bump functions shows that it is surjective.
 We already know it to be continuous.
 Therefore all that remains is to show that it is open.

 Now \(L\R\) sits inside \(L_{H,b}\R\) as a topological subspace.
 Therefore we can consider the quotient with its quotient topology.
 We can also consider this quotient as the inductive limit of the family:
 \[
  \{L_{F,b}\R / L\R : F \in \m{F}(H)\}.
 \]
 It is not hard to see that these two topologies are the same using the universal property of inductive limits (which includes quotients).
 Chasing this around shows that with both topologies a map from \(L_{H,b} \R/L\R\) is continuous if and only if it induces a continuous map from each \(L_{F,b}\R\); applying this to the identity map shows that the topologies are the same.

 Now for \(F \in \m{F}(H)\), \(L_{F,b}\R\) and
  \(\sum_{t \in F} \R^\N\)
 are Fr\'echet spaces and \(\lambda_F\) is a continuous surjection.
 It is therefore open by Banach's homomorphism theorem.
 The kernel is clearly \(L\R\) and so \(\lambda_F\) induces an isomorphism:
 \[
  L_{F,b}\R / L\R \to \sum_{t \in F} \R^\N.
 \]
 Hence \(L_{H,b}\R / L\R\) is isomorphic to the inductive limit of the spaces \(\sum_{t \in F} \R^\N\).
 This is precisely \(\sum_{t \in H} \R^\N\).
\end{proof}

The quotient map \(L_{H,b} \R \to \sum_H \R^\N\) does not split.
However if we further project \(\sum_H \R^\N\) to \(\sum_H \R\) using the first\hyp{}term projection \(\R^\N \to \R\) then we do get a splitting map.

\begin{lemma}
 The map \(L_{H,b} \R \to \sum_H \R\), \(\gamma \mapsto (\gamma'_+(t) - \gamma'_-(t))\), splits.
\end{lemma}

\begin{proof}
 Let \(\alpha_0\) be a loop in \R with a single break which is at  \(0\) such that \(\alpha'_{0,+}(0)  = 1\) and \(\alpha_{0,-}'(0) = 0\).
 For \(t \in S^1\), let \(\alpha_t\) be the result of rotating \(\alpha_0\) so that the break lies at \(t\).
 Regard an element of \(\sum_H \R\) as an \(H\)\enhyp{}indexed family of real numbers, all but a finite number of which vanish.
 Define \(\sum_H \R \to L_{H,b} \R\) by \((\nu_t) \mapsto \sum \nu_t \alpha_t\).
 This is continuous as any linear map from \(\sum_H \R\) is continuous.
 The composition \(\sum_H \R \to L_{H,b} \R \to \sum_H \R\) is easily seen to be the identity.
 We therefore have the required splitting.
\end{proof}

Using these two quotient maps we can deduce facts about the larger space from the quotient spaces.
Let us start with the positive result.

\begin{corollary}
 The space \(L_{H,b} \R^n\) is complete.
\end{corollary}

\begin{proof}
For this we use the first quotient mapping.
\[
  L \R \hookrightarrow L_{H,b} \R \twoheadrightarrow \sum_{t \in H} \R^\N
\]
where the first map is a topological embedding and the second an open surjection (quotient).
Moreover, the first and third spaces are complete.

Let us generalise this to ease the notation.
Suppose we have a short exact sequence
\[
  X \xrightarrow{i} Y \xrightarrow{q} Z
\]
where \(X\), \(Y\), and \(Z\) are locally convex topological vector spaces, \(i\) is a topological embedding, \(q\) a quotient map (whence open), and \(X\) and \(Z\) are complete.

By taking duals and adjoints we obtain a sequence
\[
  Z' \xrightarrow{q'} Y' \xrightarrow{i'} X'.
\]
Let us show that this is algebraically exact (that is, we shan't concern ourselves with topologies).
The Hahn\hyp{}Banach theorem shows that \(i' \colon Y' \to X'\) is surjective.
That \(q' \colon Z' \to Y'\) is injective is a direct consequence of the surjectivity of \(q \colon Y \to Z\).
For \(g \in Z'\), \(i'q'g\) is the linear functional \(X \to \R\) given by \(x \mapsto g q i (x)\).
Since \(q i(x) = 0\) and \(g\) is linear, this is zero.
Hence \(i'q' = 0\).
Finally, if \(i' f = 0\) then \(f \restrict_{i(X)} = 0\).
Define \(\hat{f} \colon Z \to \R\) by \(\hat{f}(z) = f(y)\) where \(q(y) = z\).
If \(y'\) is another choice of lift then \(y - y' = i(x)\) for some \(x \in X\) whence \(f(y) = f(y')\).
Clearly \(\hat{f}\) is linear and \(q' \hat{f} = f\).
This also demonstrates that it is continuous since \(f = \hat{f} q\) is continuous and \(q\) is a quotient mapping.
We therefore have exactness at \(Y'\).

Now we consider completeness.
For a contradiction, suppose that \(Y\) is not complete.
Let \(\hat{Y}\) be its completion and \(j \colon Y \to \hat{Y}\) the canonical embedding.
As \(Z\) is complete, there is a continuous linear map \(\hat{q} \colon \hat{Y} \to Z\) such that \(\hat{q} j = q\).
By assumption, \(Y\) is not complete so there is some \(y_1 \in \hat{Y} \ssetminus j(Y)\).
As \(q \colon Y \to Z\) is surjective, there is some \(y_2 \in Y\) such that \(q(y_2) = \hat{q}(y_1)\).
Let \(y_0 = y_1 - j(y_2)\).
Then \(y_0 \in \hat{Y} \ssetminus j(Y)\) and \(\hat{q}(y_0) = 0\).
As \(i\) and \(j\) are topological embeddings, the composition \(j i \colon X \to \hat{Y}\) is also a topological embedding.
As \(X\) is complete, \(j i (X)\) is closed in \(\hat{Y}\).
Since \(y_0 \notin j(Y)\), \(y_0 \notin j i(X)\) and so by the Hahn\hyp{}Banach theorem there is some \(g \in \hat{Y}'\) such that \(g(y_0) \ne 0\) and \(g \restrict_{j i(X)} = 0\).
Consider \(g j \in Y'\).
Since \(g \restrict_{j i (X)} = 0\), \(i'(g j) = 0\).
Hence \(g j = q' f\) for some \(f \in Z'\).
Consider \(g - f \hat{q} \in Y'\).
For \(y \in Y\),
\[
  (g - f \hat{q})j(y) = g j(y) - f \hat{q} j (y) = g j (y) - f q (y) = 0.
\]
However,
\[
  (g - f \hat{q})(y_0) = g(y_0) - f \hat{q}(y_0) = g(y_0) \ne 0.
\]
Hence \(g\) is a non\hyp{}zero linear functional on \(\hat{Y}\) with \(g \restrict_{j(Y)} = 0\).
Thus \(j(Y) \subseteq \ker g \ne \hat{Y}\) contradicting the fact that \(Y\) is dense in its completion.

Thus \(Y\) is complete.
\end{proof}

Using the splitting, we can deduce some negative results.

\begin{corollary}
 \label{cor:lctvsneg}
 For \(H\) infinite, the space \(L_{H,b} \R^n\) is not metrisable.
 If \(H\) is uncountable it is neither nuclear, separable, nor Lindel\"of.
\end{corollary}

\begin{proof}
 The properties of being separable and Lindel\"of are preserved by quotients.
 For locally convex topological spaces, metrisability and nuclearity are preserved by (separated) quotients; see \cite[4.2.3, 21.2.3]{hj} and \cite[I.6.3, III.7.4]{hs}.
 Or for metrisability and nuclearity we could use the fact that these are inherited by subspaces.
 Either way, it is sufficient to prove that \(\sum_H \R\) does not have the stated properties.

 For metrisable we note that as \(H\) is infinite, \(\sum_{H} \R\) contains as a topological subspace the non\hyp{}metrisable space \(\R^{(\N)}\) of all finite sequences.

 For nuclear, separable, and Lindel\"of we assume that \(H\) is uncountable.
 The proof for nuclearity reduces to that of separability using the result, as stated in \cite[3.1.6]{ap}, that if \(T \colon E \to F\) is a continuous linear map from a nuclear space to a normable space then the range of \(T\) must be separable.
 Therefore we look for a norm on \(\sum_H \R\) such that the resulting normed vector space is not separable.
 As a normed vector space is metrisable, such a topology will not be Lindel\"of either.
 The identity map on \(\sum_H \R\) will automatically be continuous from the usual topology to the norm topology, whence we deduce that \(\sum_H \R\) is neither nuclear, separable, nor Lindel\"of.

 A suitable norm on \(\sum_{H} \R\) is given by
  \(\norm[(a_t)]_1 \coloneqq \sum \abs{a_t}\).
 This is well\hyp{}defined as there are only finitely many non\hyp{}zero terms in this sum.
 It is clearly a norm.
 For \(s \in H\) let \(x_s \in \sum_{H} \R\) be the vector with a \(1\) in the \(s\)\enhyp{}place and zero elsewhere.
 We have
  \(\norm[x_s - x_r]_1 = 2\)
 and hence there are pairwise disjoint \(\norm_1\)\enhyp{}open sets \(W_s\) with \(x_s \in W_s\).
 Hence \(\sum_{H} \R\) is not separable with the \(\norm_1\)\enhyp{}topology.
\end{proof}

\subsubsection{The \(\ci\)\enhyp{}Topology}

We now turn to considering the smooth structure of \(L_{H,b} \R\).
The smooth topology on a locally convex topological vector space is the inductive topology for the smooth curves.
That is, a set is \(\ci\)\enhyp{}open if and only if its preimage under every smooth curve is open in \R. This is the topology we impose on a locally convex topological vector space when we wish to do calculus.
There are two important things to note about this topology.
Firstly, we need to start with the locally convex topology to define the smooth curves.
Therefore the \(\ci\)\enhyp{}topology depends on the locally convex one.
Secondly, it may not itself be a locally convex topology, or even a topological vector space topology.
In fact, for \(L_{H,b}\R\) it is neither by \cite[I.4.26]{akpm}.
Nonetheless, we are able to identify this topology.

\begin{proposition}
 The \(\ci\)\enhyp{}topology on \(L_{H,b}\R\) is the inductive topology from the family
  \(\{L_{F,b}\R : F \in \m{F}(H)\}\)
 in the category of topological spaces.
\end{proposition}

\begin{proof}
 From \cite[I.4.28]{akpm} we see that the \(\ci\)\enhyp{}topology on \(L_{F,b}\R\) is the trace of the \(\ci\)\enhyp{}topology on \(L_{H,b}\R\) (note that if a subspace is closed for the locally convex topology then it is closed for the \(\ci\)\enhyp{}topology).
 As \(L_{F,b}\R\) is a Fr\'echet space its \(\ci\)\enhyp{}topology agrees with its locally convex topology, \cite[I.4.11]{akpm}.
 Hence the inclusions
  \(L_{F,b}\R \to L_{H,b}\R\)
 are continuous for the \(\ci\)\enhyp{}topology on the target.
 Thus the \(\ci\)\enhyp{}topology on \(L_{H,b}\R\) is at least as coarse as the inductive topology as a topological space.

 Let
  \(U \subseteq L_{H,b}\R\)
 be open for the inductive topology.
 By definition, therefore, for each \(F \in \m{F}(H)\), \(U \cap L_{F,b}\R\) is open.
 Let
  \(c \colon \R \to L_{H,b}\R\)
 be a smooth curve and assume without loss of generality that \(c^{-1}(U)\) is not empty.
 Let \(t \in c^{-1}(U)\).
 Let \(I \subseteq \R\) be a bounded open neighbourhood of \(t\).
 Let \(K \subseteq \R\) be a compact set containing \(I\).
 As \(c\) is smooth, it is continuous and hence \(c(K)\) is compact.
 By the characterisation of bounded subsets of \(L_{H,b}\R\) there is some finite \(F \subseteq H\) such that
  \(c(K) \subseteq L_{F,b}\R\).
 Moreover,
  \(\tilde{c} \colon I \to L_{F,b} \R\),
 the restriction of \(c\), is smooth.
 Thus
 \(\tilde{c}^{-1}(U \cap L_{F,b} \R)\)
 is an open neighbourhood of \(t\).
 This is contained in \(c^{-1}(U)\) whence, as \(t\) was arbitrary, \(c^{-1}(U)\) is open.
 Thus \(U\) is \(\ci\)\enhyp{}open and so the \(\ci\)\enhyp{}topology agrees with the inductive topology.
\end{proof}

\begin{corollary}
 If \(H\) is countable, the \(\ci\)\enhyp{}topology on \(L_{H,b}\R\) is separable and Lindel\"of.
\end{corollary}

\begin{proof}
 The proof is the same as that for the locally convex topology.
\end{proof}

It is simple to deduce from this that the \(\ci\)\enhyp{}topology is not metrisable for any infinite \(H\).

\begin{corollary}
 For \(H\) infinite, the \(\ci\)\enhyp{}topology on \(L_{H,b}\R\) is not metrisable.
\end{corollary}

\begin{proof}
 Firstly we observe that if \(H_1 \subseteq H_2\) then the inclusion \(L_{H_1,b} \R \to L_{H_2,b} \R\) is a topological embedding for the \(\ci\)\enhyp{}topologies on both.
 Thus it is sufficient to prove this for \(H\) countable.
 If the \(\ci\)\enhyp{}topology on \(L_{H,b}\R\) were metrisable then it would be second countable as it is separable.
 Since the \(\ci\)\enhyp{}topology is finer than the locally convex topology, the locally convex topology would then be second countable and thus, as it is a regular Hausdorff topology, metrisable.
 This contradicts corollary~\ref{cor:lctvsneg}.
 Hence the \(\ci\)\enhyp{}topology on \(L_{H,b} \R\) for \(H\) infinite is not metrisable.
\end{proof}

We would like to deduce, again for \(H\) countable, that \(L_{H,b}\R\) is smoothly paracompact.
For convenience in the following discussion we quote two results from \cite{akpm}.

\begin{lemma}[{\cite[III.16.6]{akpm}}]
 \label{lem:cireglim}
 Let \(E\) be the strict inductive limit of a sequence of \(\Ci\)\enhyp{}normal convenient vector spaces \(E_n\) such that \(E_n \to E_{n+1}\) is closed and has the extension property for smooth functions.
 Then \(E\) is \(\Ci\)\enhyp{}regular.
 \noproof
\end{lemma}

\begin{theorem}[{\cite[III.16.10]{akpm}}]
 \label{th:cipara}
 If \(X\) is Lindel\"of and \(\m{S}\)\enhyp{}regular, then \(X\) is \(\m{S}\)\enhyp{}paracompact.
 In particular, all nuclear Fr\'echet spaces and strict inductive limits of sequences of such spaces are \(\Ci\)\enhyp{}paracompact.
 \noproof
\end{theorem}

Here, \(\m{S}\) is a subalgebra of the algebra of continuous functions on \(X\) satisfying a certain condition.
This condition is spelt out in the remark following the statement of this theorem: for each \(g \in \m{S}\) there exists an
 \(h \colon \R \to [0,1]\)
with \(h \circ g \in \m{S}\), \(h(t) = 0\) for \(t \le 0\), and \(h(t) = 1\) for \(t \ge 1\).

Theorem~\ref{th:cipara} appears to cover our situation as we have a strict inductive limit of nuclear Fr\'echet spaces.
However careful examination of this part of the proof shows that it relies on lemma~\ref{lem:cireglim} for the smooth regularity of the limit.
Therefore we need to assume that the limit is such that we have the extension property at each stage.
There is a classic example expounded in \cite[V.21.5ff]{akpm} of spaces which do not have this extension property and this example is easily modified to our spaces.
Therefore we need to adapt \cite[III.16.6]{akpm} to our situation whereupon we can use \cite[III.16.10]{akpm} to deduce that \(L_{H,b}\R\) is smoothly paracompact.

\begin{lemma}
 \label{lem:cireglimext}
 Let \(E\) be the strict inductive limit of a sequence of convenient vector spaces \(E_n\) such that \(E_n \to E_{n+1}\) is closed.
 Let \(\m{S}_n\) be the algebra consisting of those smooth functions on \(E_n\) which extend to a smooth function on \(E\).
 If each \(E_n\) is \(\m{S}_n\)\enhyp{}normal then \(E\) is \(\Ci\)\enhyp{}regular.
\end{lemma}

\begin{proof}
 This is merely a matter of replacing a few symbols in the proof of \cite[III.16.6]{akpm}.
 For convenience, we carry this out.
 We intentionally keep the same notation and language to highlight the necessary changes.

 Let \(U\) be open in \(E\) and \(0 \in U\).
 Then
  \(U_n \coloneqq U \cap E_n\)
 is open in \(E_n\).
 We choose inductively a sequence of functions \(f_n \in \m{S}_n\) such that
  \(\supp(f_n) \subseteq U_n\),
  \(f_n(0) = 1\), and
  \(f_n \restrict\!\! E_{n-1} = f_{n-1}\).
 If \(f_n\) is already constructed, we may choose by \(\m{S}_{n+1}\)\enhyp{}normality a function
  \(g \colon E_{n+1} \to \R\)
 with \(g \in \m{S}_{n+1}\),
  \(\supp(g) \subseteq U_{n+1}\),
 and
  \(g \restrict_{\supp(f_n)} = 1\).
 Since \(f_n \in \m{S}_n\), it extends to a function in \(\Ci(E,\R)\).
 This in turn restricts to an element \(\tilde{f_n}\) of \(\m{S}_{n+1}\) which, by construction, itself restricts to \(f_n\) on \(E_n\).
 As \(\m{S}_{n+1}\) is an algebra,
  \(f_{n+1} \coloneqq g \cdot \tilde{f_n}\)
 has the required properties.

 The rest of the proof proceeds unaltered.
 Now we define \(f \colon E \to \R\) by
  \(f \restrict\!\! E_n \coloneqq f_n\)
 for all \(n\).
 It is smooth since any \(c \in \Ci(\R,E)\) locally factors to a smooth curve into some \(E_n\) by \cite[(1.8)]{akpm} since a strict inductive limit is regular by \cite[(52.8)]{akpm}, so \(f \circ c\) is smooth.
 Finally, \(f(0) = 1\), and if \(f(x) \ne 0\) then \(x \in E_n\) for some \(n\), and we have \(f_n(x) = f(x) \ne 0\), thus
  \(x \in U_n \subseteq U\).
\end{proof}

\begin{proposition}
\label{prop:smreg}
 Let \(H \subseteq S^1\) be countable.
 The space \(L_{H,b} \R\) is smoothly regular.
\end{proposition}

\begin{proof}
 We will use lemma~\ref{lem:cireglimext}.
 To do so we need to show that \(L_{F,b}\R\) is \(\m{S}_F\)\enhyp{}normal where \(\m{S}_F\) is the algebra of smooth functions which extend to smooth functions on \(L_{H,b}\R\).
 We start by showing that it is \(\m{S}_F\)\enhyp{}regular.
 From this we will use \cite[III.16.10]{akpm} to deduce that it is \(\m{S}_F\)\enhyp{}paracompact, whence from \cite[III.16.2]{akpm} it is \(\m{S}_F\)--normal as required.

 The \(\ci\)\enhyp{}topology on \(L_{F,b}\R\) agrees with the locally convex topology as it is a Fr\'echet space, \cite[I.4.11]{akpm}.
 We have already shown that this is the topology inherited by \(L_{F,b}\R\) from its inclusion in \(L_{H,b} \R\) where the latter is given its locally convex topology.
 This topology is nuclear and so is defined by Hilbertian semi\hyp{}norms.
 The square of such a norm is smooth by \cite[III.13.10]{akpm}.
 Let \(U\) be a \(0\)\enhyp{}neighbourhood in \(L_{F,b}\R\), then there is some \(0\)\enhyp{}neighbourhood \(V\) in \(L_{H,b}\R\) with
  \(V \cap L_{F,b}\R \subseteq U\).
 We can thus find a Hilbertian semi\hyp{}norm
  \(p \colon L_{H,b}\R \to \R\)
 such that
  \(p^{-1}([0,1]) \subseteq V\).
 Composition of \(p^2\) with a suitable bump function on \R results in a smooth function
  \(f \colon L_{H,b} \R \to \R\)
 with support in \(V\).
 The restriction of \(f\) to \(L_{F,b}\R\) is thus in \(\m{S}_F\) and has support in \(U\).
 Hence \(L_{F,b}\R\) is \(\m{S}_F\)\enhyp{}regular.

 To apply \cite[III.16.10]{akpm} we need to show that \(\m{S}_F\) satisfies the required condition, namely that for each \(g \in \m{S}_F\) there exists an
  \(h \colon \R \to [0,1]\)
 with
  \(h \circ g \in \m{S}_F\),
  \(h(t) = 0\) for \(t \le 0\), and \(h(t) = 1\) for \(t \ge 1\).
 We will actually show that if \(f \in \m{S}_F\) and \(h \in \Ci(\R,\R)\) then
  \(h_* f \coloneqq h \circ f \in \m{S}_F\).
 Let \(f \in \m{S}_F\) and \(h \in \Ci(\R,\R)\); as \(f\) is smooth,
  \(h_* f \colon L_{F,b}\R \to \R\)
 is smooth, therefore we just need to show that it extends to a smooth function on \(L_{H,b}\R\).
 Let
  \(\tilde{f} \colon L_{H,b}\R \to \R\)
 be an extension of \(f\), then \(h_* \tilde{f}\) is an extension of \(h_* f\).

 We now deduce that \(L_{F,b}\R\) is \(\m{S}_F\)\enhyp{}paracompact, whence \(\m{S}_F\)\enhyp{}normal.
 Hence by lemma~\ref{lem:cireglimext} \(L_{H,b}\R\) is \(\Ci\)\enhyp{}regular.
\end{proof}

\begin{corollary}
 Let \(H \subseteq S^1\) be countable, then \(L_{H,b} \R\) is \(\Ci\)\enhyp{}paracompact.
\end{corollary}

\begin{proof}
 As it is Lindel\"of and \(\Ci\)\enhyp{}regular, we can apply \cite[III.16.10]{akpm}.
\end{proof}

\subsubsection{Uncountability}

In proving that \(L_{H,b} \R\) is \(\Ci\)\enhyp{}paracompact we have used the countability of \(H\) at almost every stage.
We therefore cannot readily adapt it to uncountable \(H\), specifically to the case \(H = S^1\).
Thus we are forced to leave open the question as to whether or not \(L_{\psb} \R\) itself is smoothly paracompact.

We conjecture that \(L_{\psb} \R\) is not even \(\Ci\)\enhyp{}regular.
Let us explain the rationale behind this conjecture.

Recall that the direct sum \(\sum_H \R\) sits inside \(L_{\psb} \R\) as a splitting subspace.
Therefore any results which hold for \(L_{\psb} \R\) will hold for \(\sum_H \R\).
On the other hand, the space \(\R^\N\) is well\hyp{}behaved with respect to the theory of smooth spaces and so there is no reason to doubt that a technique that works for \(\sum_H \R\) will fail for \(\sum_H \R^\N\).
Similarly, although the extension
\[
  L \R \to L_{H,b} \R \to \sum_H \R^\N
\]
does not split, one would expect a general construction for \(\sum_H \R^\N\) to be modifiable to work for \(L_{H,b} \R\).

Thus a negative result for \(\sum_H \R\) certainly implies the corresponding negative for \(L_{H,b} \R\), whilst it seems reasonable that a positive result for \(\sum_H \R\) could be adapted to one for \(L_{H,b} \R\).
Thus we may turn our attention to \(\sum_H \R\), at least for the purposes of this discussion.

Here we see the difference between countable and uncountable index sets.
For \(H\) countable, \(\sum_H \R\) is nuclear and thus its locally convex topology is determined by a family of Hilbertian semi\hyp{}norms.
It is easy to build smooth bump functions from these and this is what we used in proposition~\ref{prop:smreg}.
However for \(H\) uncountable, this is no longer the case.
The locally convex topology on \(\sum_H \R\) for \(H\) uncountable is not given by Hilbertian semi\hyp{}norms but rather by norms equivalent to that on \(\ell^1(H)\).
As \(H\) is uncountable, such a norm is nowhere even G\^ateaux differentiable, see \cite[III 13.11]{akpm}.
Indeed, for any infinite \(H\), \(\ell^1(H)\) is not even \(C^1\)\enhyp{}regular, \cite[III 14.9]{akpm}.

What remains is to show that if \(\sum_H \R\) is \(\Ci\)\enhyp{}regular then a suitable smooth bump function extends to some Banach completion of \(\sum_H \R\).
The discussion at the start of \cite[III 13]{akpm} is relevant here.
There the issue of smooth semi\hyp{}norms is discussed and reason is given for considering only Banach spaces because any semi\hyp{}norm on a locally convex topological vector space defines an associated Banach space.
However if the original semi\hyp{}norm is smooth it may not be the case that its extension is everywhere smooth.

An alternative approach would be to attempt a smooth version of the Hahn\enhyp{}Banach theorem.
One could view proposition~\ref{prop:smreg} as the separable case.
By analogy, one would attempt to extract from proposition~\ref{prop:smreg} and from lemma~\ref{lem:cireglimext} the ``one\hyp{}step extension lemma'' crucial to the proof of the full Hahn\enhyp{}Banach theorem.
However, there are difficulties with showing that the resulting family satisfies the requirements of Zorn's Lemma and so adapting the proof of the Hahn\enhyp{}Banach theorem is not a simple task.

In the light of the other difficulties with the spaces of piecewise\hyp{}smooth and bounded loops\emhyp{}as detailed in the next section\emhyp{}we defer settling this conjecture to a later date.

\section{The Diffeomorphism Group}
\label{sec:diff}

In this section we examine the action of the diffeomorphism group of the circle.
This acts on the space of piecewise\hyp{}smooth bounded loops by precomposition.
We shall see that this action is fairly bad, both in terms of the continuity of the map:
\[
 \smth(S^1) \to \m{L}(L_{\psb} \R)
\]
and in terms of the continuity of the maps:
\[
 \smth(S^1) \to L_{\psb} \R, \quad \sigma \mapsto \alpha \circ \sigma
\]
for a fixed \(\alpha\).

The action does not become any nicer when restricted to the circle, acting by rigid rotation.
Therefore we also consider the possibility of improving the circle action.
However, we find that improving the circle action leads to a considerable worsening of the topology and we doubt whether the trade\hyp{}off is worthwhile.

\subsection{The Action of the Diffeomorphism Group}

There are many different topologies that one might wish to put on the space of continuous linear maps of a locally convex topological vector space.
As we are expecting negative results we shall use what is known as the \emph{weak} or \emph{simple} topology.
This is the coarsest topology that one would sanely think of using.
Our negative results will therefore propagate backwards to any other sensible topology.

Let \(E\) be a locally convex topological vector space.
Let \(\m{L}(E)\) be the space of continuous linear maps from \(E\) to itself.
The weak topology on \(\m{L}(E)\) is the topology of pointwise convergence, or uniform convergence on finite sets.
To define a \(0\)\enhyp{}basis for this topology, for \(X,Y \subseteq E\), let
 \(N(X,Y) \subseteq \m{L}(E)\)
be the set:
\[
 N(X,Y) \coloneqq \{T \in \m{L}(E) : T(X) \subseteq Y\}.
\]
Then a \(0\)\enhyp{}basis for the weak topology on \(\m{L}(E)\) is the family of those \(N(X,U)\) with \(X\) finite and \(U\) a \(0\)\enhyp{}neighbourhood.
We write \(\m{L}_s(E)\) when we wish to emphasise that we are considering \(\m{L}(E)\) with the weak, or simple, topology.

This topology is closely related to the notion of \emph{separate continuity}.
For topological spaces \(X,Y,Z\) a map
 \(f \colon X \times Y \to Z\)
is separately continuous if the maps \(x \mapsto f(x,y_0)\) and \(y \mapsto f(x_0,y)\) are continuous for all \(x_0 \in X\) and \(y_0 \in Y\).
We are interested in the special case where \(Y = Z = E\) and the maps \(y \to f(x_0,y)\) are continuous and linear.
We therefore have an induced map
 \(f^\lor \colon X \to \m{L}(E)\)
and, under these conditions, it is easy to see that the separate continuity of \(f\) is equivalent to the continuity of \(f^\lor\) with the weak topology on the target.

Let us return to the diffeomorphism group acting on piecewise\hyp{}smooth bounded loops.

\begin{proposition}
 Let
  \(\sigma \colon S^1 \to S^1\)
 be a diffeomorphism.
 Let \(H \subseteq S^1\) be a subset.
 The induced map
  \(\sigma^* \colon \gamma \mapsto \gamma \circ \sigma\)
 is a linear homeomorphism from \(L_{\sigma(H),b} \R\) onto \(L_{H,b} \R\).
\end{proposition}

\begin{proof}
 Let
  \(\gamma \in L_{F,b} \R\).
 As \(\gamma\) and \(\sigma\) are continuous, \(\gamma \circ \sigma\) is continuous.
 Since \(\gamma\) is smooth on \(S^1 \ssetminus F\), \(\gamma \circ \sigma\) is smooth on
  \(\sigma^{-1}(S^1 \ssetminus F)\).
 As \(\sigma\) is a bijection,
  \(\sigma^{-1}(S^1 \ssetminus F) = S^1 \ssetminus \sigma^{-1}(F)\)
 and \(\sigma^{-1}(F)\) is a finite subset of
  \(\sigma^{-1}(\sigma(H)) = H\).
 Hence \(\gamma \circ \sigma\) is piecewise\hyp{}smooth with breaks in \(\sigma^{-1}(F)\).

 As \(\sigma\) is a diffeomorphism on \(S^1\), each derivative is bounded.
 For each \(k \in \N\) let
  \(m_k \coloneqq \sup\{\abs{\sigma^{(j)}(t)} : t \in S^1, 0 \le j \le k\}\).
 From Fa\`a di Bruno's formul\ae\ for the chain rule for higher derivatives we see that there is some constant \(N_k\) depending only on \(k\) such that for
  \(t \in S^1 \ssetminus \sigma^{-1}(F)\):
 \[
  \abs{(\gamma \circ \sigma)^{(k)}(t)} \le N_k m_k \max\{ \abs{\gamma^{(j)}(\sigma(t))} : 0 \le j \le k\}.
 \]
 Hence:
 \begin{gather*}
  \sup\{\abs{(\gamma \circ \sigma)^{(j)}(t)} : t \in S^1 \ssetminus \sigma^{-1} (F), 0 \le j \le k\} \\
  \le N_k m_k \sup\{\abs{\gamma^{(j)}(t)} : t \in S^1 \ssetminus F, 0 \le j \le k\}.
 \end{gather*}
 Thus the derivatives of \(\gamma \circ \sigma\) are bounded on their domains of definition and so \(\gamma \circ \sigma\) is piecewise\hyp{}smooth and bounded.
 Moreover, as the left and right limits exist we have:
 \begin{gather*}
  \sup\{\abs{(\gamma \circ \sigma)_\pm^{(j)}(t)} : t \in S^1, 0 \le j \le k\} \\
  \le N_k m_k \sup\{\abs{\gamma_\pm^{(j)}(t)} : t \in S^1, 0 \le j \le k\}.
 \end{gather*}
 whence the map
  \(L_{F,b} \R^n \to L_{\sigma^{-1}(F),b} \R^n\),
  \(\gamma \mapsto \gamma \circ \sigma\),
 is continuous.
 As its inverse is
  \(\gamma \mapsto \gamma \circ \sigma^{-1}\)
 it is therefore a linear homeomorphism.

 Using the characterisation of \(L_{H,b} \R^n\) as an inductive limit we deduce that \(\sigma\) induces a linear homeomorphism of \(L_{\sigma(H),b} \R^n\)
 onto \(L_{H,b} \R^n\) for any \(H\).
\end{proof}

We shall shortly see that this is the best statement that can be made about this action.
To continue our analysis we need some suitable open sets in \(L_{\psb}\R\).
We will use these to define open sets of the form \(N(\{\alpha\}, V)\) which will separate our diffeomorphisms.
Let
 \(\sign \colon \R \ssetminus \{0\} \to \{-1, 1\}\)
be the sign function,
 \(\sign(x) = x / \abs{x}\).

\begin{lemma}
 \label{lem:sign}
 Let
  \(W \subseteq L_{\psb}\R\)
 be the set of all piecewise\hyp{}smooth loops with first derivative bounded away from \(0\) (on its domain of definition).
 Define a relation on \(W\) by: \(\alpha \sim \beta\) if
  \(\sign(\alpha'(t)) = \sign(\beta'(t))\)
 for all \(t\) where both sides are defined.
 Define another relation by \(\alpha \approx \beta\)
 if \(\alpha \sim \beta\) or \(\alpha \sim - \beta\).
 Let
  \(\sigma \colon S^1 \to S^1\)
 be a diffeomorphism.
 \begin{enumerate}
 \item
   \(W\) is a non\hyp{}empty \(\sigma^*\)\enhyp{}invariant open set.

 \item
  The relations \(\sim\) and \(\approx\) are equivalence relations.

 \item
  The equivalence classes of \(\sim\) are open convex cones.
  Each equivalence class of \(\approx\) is the union of two equivalence classes of \(\sim\).

 \item
  The equivalence classes of \(\sim\) are indexed by two copies of the family of finite subsets of \(S^1\) of even size.
  Those of \(\approx\) are indexed by the family of finite subsets of \(S^1\) of even size.

 \item
  The diffeomorphism \(\sigma\) induces a permutation on the equivalence classes of \(\sim\) and \(\approx\).
  If \(V(F)\) is the \(\approx\)\enhyp{}equivalence class corresponding to a finite subset \(F \subseteq S^1\) then
   \(\sigma^* V(F) = V(\sigma^{-1}(F))\).
 \end{enumerate}
\end{lemma}

\begin{proof}
 \begin{enumerate}
 \item
  Let \(\alpha \in W\).
  Then there is some \(M \ge 0\) such that
   \(\abs{\alpha'(t)} \ge M\)
  for all \(t\) where it is defined.
  Let
   \(\beta \in \alpha + U((1, M/2))\).
  Then
   \(\abs{\beta'(t) - \alpha'(t)} < M/2\)
  wherever both are defined.
  Hence
   \(\abs{\beta'(t)} \ge M/2\)
  for all \(t\) where both \(\beta'(t)\) and \(\alpha'(t)\) are defined.
  As this is a finite subset of the domain of definition of \(\beta'\) and \(\beta'\) is continuous on this domain, we must have
   \(\abs{\beta'(t)} \ge M/2\)
  wherever \(\beta'(t)\) is defined.
  Hence \(\beta \in W\).

  Again let \(\alpha \in W\).
  By the chain rule,
   \((\sigma^* \alpha)'(t) = \alpha'(\sigma(t)) \sigma'(t)\)
  wherever this is defined.
  As \(\sigma\) is a diffeomorphism, \(\sigma'(t)\) is bounded away from \(0\).
  By assumption, \(\alpha'(\sigma(t))\) is bounded away from zero.
  Hence
   \(\sigma^* \alpha \in W\).

 \item
  Start with the relation \(\sim\).
  Reflexivity and symmetry are straightforward.
  For transitivity the only difficulty is the domain of definition.
  Firstly note that for \(\alpha \in W\), \(\sign(\alpha'(t))\) is defined whenever \(\alpha'(t)\) is defined as it is bounded away from zero.
  Also
   \(t \mapsto \sign(\alpha'(t))\)
  is a locally constant function on the domain of \(\alpha'\), which is an open subset of \(S^1\).
  Now from \(\alpha \sim \beta\) and \(\beta \sim \gamma\) we readily deduce that
   \(\sign(\alpha'(t)) = \sign(\gamma'(t))\)
  for all but a finite number of points where both sides make sense.
  But then we can extend this to those points since both sides are continuous and constant in a neighbourhood of each missing point.
  Thus
   \(\sign(\alpha'(t)) = \sign(\gamma'(t))\)
  wherever both sides are defined.

  The properties of \(\approx\) follow almost immediately from those for \(\sim\).
  The only extra fact we need is the obvious one that \(\alpha \sim - \beta\) if and only if \(- \alpha \sim \beta\).

 \item
  To show that an equivalence class of \(\sim\) is open we need to show that for each \(\beta \in W\) there is an open neighbourhood \(V\) of \(\beta\) such that if \(\gamma \in V\) then \(\beta \sim \gamma\).
  So let \(\beta \in W\).
  Then there is some \(K > 0\) such that
   \(\abs{\beta'(t)} \ge K\)
  whenever it is defined.
  Let \(\gamma \in W\) be such that
   \(\beta - \gamma \in U((1,K))\).
  Then
   \(\abs{\beta'(t) -\gamma'(t)} < K\)
  whenever both are defined.
  Hence as
   \(\abs{\beta'(t)} \ge K\),
   \(\gamma'(t)\) and \(\beta'(t)\) have the same sign.
  Thus \(\beta \sim \gamma\).

  For the cone, let \(\lambda, \mu > 0\) and \(\alpha, \beta \in W\) be such that \(\alpha \sim \beta\).
  Then for all \(t\) where both are defined, \(\alpha'(t)\) and \(\beta'(t)\) are either both positive or both negative.
  Thus:
  \[
   \abs{\lambda \alpha'(t) + \mu \beta'(t)} = \lambda \abs{\alpha'(t)} + \mu \abs{\beta'(t)}.
  \]
  Since both \(\alpha'(t)\) and \(\beta'(t)\) are bounded away from \(0\) there is some \(K > 0\) such that
   \(\abs{\alpha'(t)} \ge K\)
  and
   \(\abs{\beta'(t)} \ge K\).
  Then:
  \[
   \abs{\lambda \alpha'(t) + \mu \beta'(t)} \ge (\lambda + \mu) K
  \]
  and
   \((\lambda + \mu) K > 0\).
  Hence
   \(\lambda \alpha + \mu \beta \in W\).
  Moreover,
   \(\lambda \alpha'(t) + \mu \beta'(t)\)
  has the same sign as, say, \(\alpha'(t)\) and thus
   \(\alpha \sim \lambda \alpha + \mu \beta\).

  Finally, it is obvious that the \(\approx\)\enhyp{}equivalence class of \(\alpha\) is the union of the \(\sim\)\enhyp{}equivalence classes of \(\alpha\) and \(-\alpha\).

 \item
  Let \(\alpha \in W\).
  The \(\sim\)\enhyp{}equivalence class of \(\alpha\) is clearly completely determined by the points in \(S^1\) where \(\alpha'\) changes sign.
  As \(\alpha'\) is bounded away from zero these points must be a subset of the breaks of \(\alpha\) and hence a finite subset of \(S^1\).
  Moreover, there must be an even number as we have to return to our starting point after a circuit of \(S^1\).
  For any finite subset of \(S^1\) of even size we can find an \(\alpha\) corresponding to this set.
  That there are two copies comes from the fact that \(\alpha\) and \(-\alpha\) have the same set of sign\hyp{}shifts.
  This doubling disappears when we consider the \(\approx\)\enhyp{}equivalence classes.

 \item
  The \(\sim\)\enhyp{}equivalence classes are open convex cones in \(W\).
  As they are cones, they are path\hyp{}connected.
  As they partition \(W\) each is the complement (in \(W\)) of the others whence is closed.
  Thus they are the connected components of \(W\).
  As \(\sigma^*\) is a self\hyp{}homeomorphism of \(L_{\psb}\R\) it induces a self\hyp{}homeomorphism of \(W\) and thus a permutation on the connected components.
  Moreover, as \(\sigma^*\) is linear we see that
   \(\sigma^* \alpha \sim - \sigma^* \beta\)
  if and only if \(\alpha \sim - \beta\) whence \(\sigma^*\) induces a permutation on the \(\approx\)\enhyp{}equivalence classes.

  By the chain rule,
   \((\sigma^* \alpha)'(t) = \alpha'(\sigma(t)) \sigma'(t)\).
  As \(\sigma\) is a diffeomorphism, \(\sigma'\) is mono\hyp{}signed.
  Therefore \((\sigma^* \alpha)'\) has a sign\hyp{}change at \(t_0\) if and only if \(\alpha'\) has a sign\hyp{}change at \(\sigma(t_0)\).
  Hence if \(V(F)\) is the \(\approx\)\enhyp{}equivalence class corresponding to \(F \subseteq S^1\),
   \(\sigma^* V(F) = V(\sigma^{-1}(F))\).
  \qedhere
 \end{enumerate}
\end{proof}

The \(\approx\)\enhyp{}equivalence classes are very useful open sets for examining the action of the diffeomorphism group.

\begin{proposition}
 \label{prop:disconnect}
 Let
  \(\m{A} \subseteq \smth(S^1)\)
 be a family of diffeomorphisms for which there is some finite \(F_0 \subseteq S^1\) of even size with the property that if
  \(\sigma, \tau \in \m{A}\)
 are such that
  \(\sigma^{-1}(F_0) = \tau^{-1}(F_0)\)
 then \(\sigma = \tau\).

 For any \(\alpha \in V(F_0)\) there is a family of pairwise disjoint open sets in \(E\) which covers the set
  \(\{\sigma^* \alpha : \sigma \in \smth(S^1)\}\)
 with the property that each member of this family contains at most one element of
  \(\{\sigma^* \alpha : \sigma \in \m{A}\}\).
\end{proposition}

\begin{proof}
 Consider the family
  \(\m{V} \coloneqq \{V(F) : \abs{F} = \abs{F_0}\}\).
 These are pairwise disjoint open sets in \(E\).
 From lemma~\ref{lem:sign} we know that
  \(\sigma^* \alpha \in V(\sigma^{-1}(F_0))\)
 for
  \(\sigma \in \smth(S^1)\).
 As \(\sigma\) is a bijection,
  \(\abs{\sigma^{-1}(F_0)} = \abs{F_0}\)
 and hence
  \(V(\sigma^{-1}(F_0)) \in \m{V}\).
 Thus \(\m{V}\) covers the set
  \(\{\sigma^* \alpha : \sigma \in \smth(S^1)\}\))
 as required.

 Now if \(\sigma\) and \(\tau\) are distinct elements of \(\m{A}\) we know that
  \(\sigma^{-1}(F_0) \ne \tau^{-1}(F_0)\).
 Hence \(V(\sigma^{-1}(F_0))\) and \(V(\tau^{-1}(F_0))\) are distinct.
 Since
  \(\sigma^* \alpha \in V(\sigma^{-1}(F_0))\)
 we deduce that each element of \(\m{V}\) can contain at most one element of the set
  \(\{\sigma^* \alpha : \sigma \in \m{A}\}\).
\end{proof}

From this technical result we can determine just how bad is the action of the diffeomorphism group, and even of the circle acting by rigid rotations.

\begin{corollary}
 \begin{enumerate}
 \item
  Let \(\m{A}\), \(F_0\), and \(\alpha\) be as in the statement of proposition~\ref{prop:disconnect}.
  Then the set
   \(\{\sigma^* \alpha : \sigma \in \m{A}\}\)
  is discrete in \(L_{\psb}\R\).

 \item
  Let \(\m{A}\) satisfy the conditions of proposition~\ref{prop:disconnect}.
  Then the set
   \(\{\sigma^* : \sigma \in \m{A}\}\)
  is discrete in \(\m{L}_s(L_{\psb}\R)\).

 \item
  The set
   \(\{\sigma^* : \sigma \in \smth(S^1)\}\)
  is totally disconnected in \(\m{L}_s(L_{\psb}\R)\).

 \item
  For \(t \in S^1\) let
   \(\rho_t \colon S^1 \to S^1\)
  be rotation by \(t\).
  The family
   \(\{\rho_t : t \in S^1\}\)
  satisfies the conditions of proposition~\ref{prop:disconnect} for any non\hyp{}periodic finite subset \(F_0 \subseteq S^1\) of even size.
 \end{enumerate}
\end{corollary}

To say that a set \(F \subseteq S^1\) is periodic means that there is some \(t \in (0,1)\) such that \(F + t = F\).
For \(F\) finite this is equivalent to saying that \(F\) is the union of a finite number of cosets of
 \(C_n \coloneqq \{k/n : 0 \le k < n\}\),
for some \(n\).

\begin{proof}
 \begin{enumerate}
 \item
  By proposition~\ref{prop:disconnect} for \(\tau \in \m{A}\) the set \(\{\tau^* \alpha\}\) is open in
   \(\{\sigma^*\alpha : \sigma \in \smth(S^1)\}\).

 \item
  Let \(F_0\) and \(\alpha\) be as in the statement of proposition~\ref{prop:disconnect}.
  From the proof of proposition~\ref{prop:disconnect}, the sets
   \(\{N(\{\alpha\}, V(F)) : \abs{F} = \abs{F_0}\}\)
  are pairwise disjoint sets, open in \(\m{L}_s(L_{\psb}\R)\) which cover the set
   \(\{\sigma^* : \sigma \in \smth(S^1)\}\).
  Moreover, each \(N(\{\alpha\}, V(F))\) can contain at most one element of the set
   \(\{\sigma^* : \sigma \in \m{A}\}\).

 \item
  Let \(\sigma, \tau\) be distinct diffeomorphisms.
  Let
   \(\m{A} = \{\sigma, \tau\}\).
  As they are distinct, there is some \(t_0 \in S^1\) such that
   \(\sigma^{-1}(t_0) \ne \tau^{-1}(t_0)\).
  Choose another point \(t_1 \in S^1\) such that
   \(\sigma^{-1}(t_1) \ne \tau^{-1}(t_0)\).
  Let \(F_0 = \{t_0, t_1\}\).
  The pair \((\m{A}, F_0)\) satisfy the conditions of proposition~\ref{prop:disconnect}.
  Thus we have a family of pairwise disjoint open sets \(\m{V}\) in \(\m{L}_s(L_{\psb}\R)\) with the property that \(\sigma^*\) and \(\tau^*\) lie in two distinct ones.
  Divide the family \(\m{V}\) into two parts such that the set containing \(\sigma^*\) is in one part and that containing \(\tau^*\) in the other.
  Let \(U\) and \(V\) be the unions of the sets in these two parts.
  Then \(U\) and \(V\) are disjoint open sets in \(\m{L}_s(L_{\psb}\R)\) such that \(U \cup V\) contains
   \(\{\rho^* : \rho \in \smth(S^1)\}\)
  and \(\sigma^* \in U\), \(\tau^* \in V\).
  Hence
   \(\{\rho^* : \rho \in \smth(S^1)\}\)
  is totally disconnected.

 \item
  Let \(F_0\) be a non\hyp{}empty, non\hyp{}periodic subset of \(S^1\) of finite even size.
  Let \(s, t \in S^1\) be such that
   \(\rho_t(F_0) = \rho_s(F_0)\).
  Then
   \(F_0 = \rho_{s - t}(F_0) = F_0 + (s - t)\).
  Hence as \(F_0\) is non\hyp{}periodic, \(s = t\).
  \qedhere
 \end{enumerate}
\end{proof}

\begin{corollary}
 The action of the circle on \(L_{\psb}\R\) is not separately continuous.
 Let \(F \subseteq S^1\) be an even, non\hyp{}empty, non\hyp{}periodic subset.
 Let \(\alpha \in V(F)\).
 Then the set
  \(\{{\rho_t}^* \alpha : t \in S^1\}\)
 is discrete in \(L_{\psb}\R\).
 \noproof
\end{corollary}

In fact we can do better than this last statement.
It is not hard to show that the map
 \(t \mapsto {\rho_t}^* \alpha\)
has discrete image if, and only if, \(\alpha\) has a genuine break.

\subsection{Fixing the Circle Action}

In the last section we saw that the topology on \(L_{\psb}\R\) is particularly ill\hyp{}behaved with regard to the action of the diffeomorphism group, and in particular with regard to the natural circle action.
Due to the importance of this circle action it is tempting to try to fix this problem.
Unfortunately we shall see in this section that there is no truly satisfying solution.

Let \(\m{T}\) be a locally convex topology on \(L_{\psb}\R\) with the following properties:
\begin{enumerate}
\item
 The maps
  \(L_{F,b}\R \to L_{\psb}\R\)
 and \(L_{\psb}\R \to L^0\R\)
 are all continuous for the topology \(\m{T}\).

\item
 The topology is Hausdorff.

\item
 The circle action is separately continuous.
\end{enumerate}

To investigate this topology we consider the following piecewise\hyp{}linear loop.
Let
 \(\alpha_0 \colon S^1 \to \R\)
be the loop
\[
 \alpha_0(t) = \begin{cases} t - \frac14 &
  0 \le t < \frac12, \\
  \frac34 - t & \frac12 \le t < 1
 \end{cases}
\]
 illustrated in figure~\ref{fig:s1odd}.

\begin{figure}
 \begin{center}
  \myfig{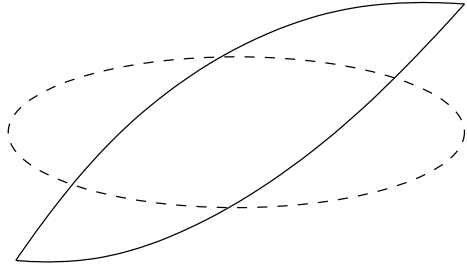}
  \caption{The Loop \(\alpha_0\)\label{fig:s1odd}}
 \end{center}
\end{figure}

Note that
 \(\alpha_0(t) + \alpha_0(t + \frac12) = 0\)
for all \(t \in S^1\).
This is a useful property so we shall give it a name.

\begin{defn}
 We call a loop
  \(\beta \colon S^1 \to \R\)
 \emph{\(S^1\)\enhyp{}odd} if it satisfies the condition
  \(\beta(t) + \beta(t + \frac12) = 0\)
 for all \(t \in S^1\).
\end{defn}

The set of such loops has certain obvious properties.

\begin{lemma}
 \begin{enumerate}
 \item
  A loop is \(S^1\)\enhyp{}odd if it is a linear combination of \(S^1\)\enhyp{}odd loops.

 \item
  A loop is \(S^1\)\enhyp{}odd if it is the rotation of an \(S^1\)\enhyp{}odd loop.

 \item
  The only \(S^1\)\enhyp{}odd constant loop is the zero loop.

 \item
  If \(\beta\) is an integrable \(S^1\)\enhyp{}odd loop then there is a unique constant \(c\) (depending on \(\beta\)) for which the loop
   \(t \mapsto \int_0^t \beta(s) d s + c\)
  is \(S^1\)\enhyp{}odd.
  This constant is
   \(-\frac12 \int_0^{\frac12} \beta(s) d s\).
  \noproof
 \end{enumerate}
\end{lemma}

Let us return to \(\alpha_0\).
By assumption the circle action is separately continuous.
Therefore the map
 \(t \mapsto R_t \alpha_0\)
is a continuous map
 \(S^1 \to (L_{\psb}\R, \m{T})\).
The image is thus compact.
Let \(A\) be the convex circled hull of this set.
By \cite[II.4.3]{hs}, this has precompact closure and hence is bounded.
Let
 \(E_A \coloneqq \bigcup_{n \in \N} n A\).
This is a subspace of \(L_{\psb}\R\) and we equip it with the norm \(\norm_A\) defined by the gauge of \(A\).
That is:
\[
 \norm[x]_A \coloneqq \inf\{\lambda > 0 : x \in \lambda A\}.
\]
The pair \((E_A, \norm_A)\) is then a normed vector space which injects continuously into \((L_{\psb}\R, \m{T})\).
We transfer our attention to \((E_A, \norm_A)\).

\begin{lemma}
 The normed vector space \((E_A, \norm_A)\) is the space of \(S^1\)\enhyp{}odd piecewise\hyp{}linear loops.
 The set
  \(\{R_t \alpha_0 : t \in [0, \frac12)\}\)
 is a basis for \(E_A\).
 The map
  \((\xi_t) \to \sum \xi_t R_t \alpha_0\)
 is an isometric isomorphism
  \((\R^{([0, \frac12))}, \norm_1) \to (E_A, \norm_A)\).
\end{lemma}

\begin{proof}
 Since \(R_t \alpha_0\) is \(S^1\)\enhyp{}odd and piecewise\hyp{}linear for any \(t \in S^1\), every element of \(E_A\) is \(S^1\)\enhyp{}odd and piecewise\hyp{}linear.
 To prove the converse, let \(\beta\) be an \(S^1\)\enhyp{}odd, piecewise\hyp{}linear loop.
 Let
  \(t_1, \dotsc, t_{k-1} \in (0,\frac12)\)
 with \(t_i < t_{i+1}\) be such that
  \(\{\frac12 - t_i : 1 \le i \le k\}\)
 is the set of breaks of \(\beta\) which lie in \((0, \frac12)\).
 Put \(t_0 = 0\) and \(t_k = \frac12\).
 Choose points
  \(s_i \in (\frac12 - t_i, \frac12 - t_{i-1})\)
 for \(1 \le i \le k\).
 Let \(V \subseteq E_A\) be the linear span of the set
  \(\{R_{t_i} \alpha_0 : 0 \le i \le k-1\}\).

 The breaks of \(\gamma \in V\) lie in the set
  \(\{\frac12 - t_i, 1 - t_i : 0 \le i \le k\}\)
 and so \(\gamma\) is differentiable at \(s_i\).
 Define a linear transformation \(D \colon V \to \R^k\) by
  \(\gamma \to (\gamma'(s_i))\).
 The image of \(R_{t_i} \alpha_0\) is the vector
  \((1, \dotsc, 1, -1, \dotsc, -1)\)
 where \(-1\) occurs with multiplicity \(i\).
 These vectors form a basis of \(\R^k\), whence \(\dim V = k\), the \(R_{t_i} \alpha_0\) are a basis for \(V\), and \(D\) is an isomorphism.

 In particular, there are \(\xi_i \in \R\), \(0 \le i \le k-1\), such that putting
  \(\gamma = \sum \xi_i R_{t_i} \alpha_0\)
 then
  \(D(\gamma) = (\beta'(s_1), \dotsc, \beta'(s_k))\).
 By construction, \(\gamma\) and \(\beta\) are both \(S^1\)\enhyp{}odd, piecewise\hyp{}linear loops with the same break points in \((0,\frac12)\) and the same derivatives on \([0, \frac12]\).
 Thus \(\beta - \gamma\) is an \(S^1\)\enhyp{}odd, piecewise\hyp{}linear loop with zero derivative on \([0, \frac12]\).
 Hence \(\beta - \gamma = 0\) whence \(\beta \in E_A\).

 This also proves the claim that
  \(\{R_t \alpha_0 : t \in [0, \frac12)\}\)
 is a basis for \(E_A\) since we cover every choice of
  \((t_1, \dotsc, t_{k-1})\)
 by this means.

 For \(\beta \in E_A\) with expansion
  \(\beta = \sum \xi_j R_{t_j} \alpha_0\)
 we have the obvious inequality
  \(\norm[\beta]_A \le \sum \abs{\xi_j}\).
 To show the converse, assume without loss of generality that \(\sum \abs{\xi_j} = 1\).
 If \(\norm[\beta]_A < 1\) then there is some \(\lambda > 1\) with \(\lambda \beta \in A\).
 Thus
  \(\lambda \beta = \sum \zeta_i R_{s_i} \alpha_0\)
 for some \(s_i \in [0, \frac12)\)
 and \(\zeta_i \in \R\) with
  \(\sum \abs{\zeta_i} \le 1\).
 As the \(R_t \alpha_0\) are a basis for \(E_A\) we must have that the two expressions for \(\beta\) are one and the same.
 Thus
  \(\sum \abs{\lambda \xi_j} = \sum \abs{\zeta_i} \le 1\).
 But
  \(\sum \abs{\lambda \xi_j} = \lambda \sum \abs{\xi_j} = \lambda > 1\).
 Thus we see that
  \(\norm[\beta]_A = 1 = \sum \abs{\xi_j}\).

 Moreover, this demonstrates that the map
  \((\xi_t) \to \sum \xi_t R_t \alpha_0\)
 is an isometric isomorphism
  \((\R^{([0, \frac12))}, \norm_1) \to (E_A, \norm_A)\).
\end{proof}

Our next task is to identify the completion of this space.
It will simplify matters if we differentiate everything involved and chose some convention for the points where our piecewise\hyp{}linear maps have breaks.
Let
 \(\beta_0 \colon S^1 \to \R\)
be the step function
\[
 \beta_0(t) = \begin{cases} 1 & 0 \le t < \frac12 \\
  -1 & \frac12 \le t < 1.
 \end{cases}
\]
Then \(\beta_0 = \alpha_0'\) at all but finitely many points.
Let \(E_B\) be the linear span of the family
 \(\{R_t \beta_0 : t \in [0, \frac12)\}\)
with norm
\[
 \norm[\sum_{j = 1}^k \xi_j R_{t_j} \beta_0]_B = \sum_{j=1}^k \abs{\xi_j}
\]
The uniqueness of \(S^1\)\enhyp{}odd integrals shows that \(S^1\)\enhyp{}odd integration defines an isometric isomorphism from \(E_B\) to \(E_A\).
To identify the completion of \(E_B\) we need to reformulate the norm in terms of values of the loop.

\begin{lemma}
 Let
  \(\gamma = \sum_{j=0}^k \xi_j R_{t_j} \beta_0\)
 and assume, without loss of generality, that \(t_j < t_{j+1}\) and \(t_0 = 0\).
 Let \(t \in [0, \frac12)\).
 Let \(i\) be the maximum of the set
  \(\{j : \frac12 - t_j > t\}\).
 Then
  \(\gamma(t) = \xi_0 + \dotsb + \xi_i - \xi_{i-1} - \dotsb \xi_k\).
\end{lemma}

\begin{proof}
 For \(s \in [0, \frac12)\) we have by definition,
  \((R_s \beta_0)(t) = \beta_0(t + s)\)
 which is \(1\) if
  \(t + s \in [0, \frac12)\)
 and is \(-1\) if
  \(t + s \in [\frac12, 1)\).
 Therefore
  \((R_s \beta_0)(t) = 1\)
 if \(t < \frac12 - s\) and
  \((R_s \beta_0)(t) = -1\)
 if \(t \ge \frac12 - s\).

 Let \(i\) be as in the statement.
 Note that \(i\) is well\hyp{}defined as
  \(\frac12 - t_0 = \frac12 > t\).
 Then for \(j \le i\), \(t < \frac12 - t_j\) so
  \((R_{t_j} \beta_0)(t) = 1\);
 whilst for \(j > i\), \(t \ge \frac12 - t_j\) so
  \((R_{t_j}\beta_0)(t) = -1\).
 Substituting in yields the given expression.
\end{proof}

\begin{corollary}
 \label{cor:xival}
 Let
  \(\gamma = \sum_{j=0}^k \xi_j R_{t_j} \beta_0\)
 with, as above, \(t_j < t_{j+1}\) and \(t_0 = 0\).
 Then
 \[
  \xi_j = \begin{cases}
   \frac12 \gamma(\frac12 - t_{j+1}) - \frac12 \gamma(\frac12 - t_j) &
   j \ne k \\
   \frac12 \gamma(0) - \frac12 \gamma(\frac12 - t_k) &
   j = k
  \end{cases}
 \]
\end{corollary}

\begin{proof}
 For
  \(j \in \{0, \dotsc, k-1\}\)
 let
  \(s_j = \frac12 - t_{j+1}\).
 Let \(s_k = 0\).
 By construction,
  \(\max\{i : \frac12 - t_i > s_j\} = j\).
 Thus
 \[
  \gamma(s_j) = \xi_0 + \dotsb + \xi_j - \xi_{j-1} - \dotsb - \xi_k.
 \]
 Hence for \(j \ge 1\),
  \(\gamma(s_j) - \gamma(s_{j-1}) = 2 \xi_j\)
 and
  \(\gamma(s_0) + \gamma(s_k) = 2 \xi_0\).
 Now \(s_k = 0\) so as \(\gamma\) is \(S^1\)\enhyp{}odd,
  \(\gamma(s_k) = - \gamma(\frac12)\).
 Thus for \(0 \le j < k\),
  \(2 \xi_j = \gamma(\frac12 - t_{j+1}) - \gamma(\frac12 - t_j)\)
 and
  \(2 \xi_k = \gamma(0) - \gamma(\frac12 - t_k)\).
\end{proof}

\begin{proposition}
 The loops in \(E_B\) are of bounded variation and the total variation of \(\gamma \in E_B\) is \(4\norm[\gamma]_B\).
\end{proposition}

\begin{proof}
 Let
  \(\gamma = \sum_{j=0}^k \xi_j R_{t_j} \beta_0\)
 be an element of \(E_B\) with \(t_j < t_{j+1}\) and \(t_0 = 0\) as before.
 Let \(\m{P}\) be a partition of \(S^1\) and assume without loss of generality that
  \(\frac12 - t_j \in \m{P}\)
 for all \(j\) and
  \(\m{P} + \frac12 = \m{P}\).
 As \(\gamma\) is constant on the intervals \([t_j, t_{j+1})\) and on \([t_k, \frac12)\) it is easy to see that the variation of \(\gamma\) with respect to the partition \(\m{P}\) is:
 \begin{gather*}
  \abs{\gamma(0) - \gamma(\frac12 - t_k)} + \sum_{j=0}^{k-1} \abs{\gamma(\frac12 - t_{j+1}) - \gamma(\frac12 - t_j)} \\
  + \abs{\gamma(\frac12) - \gamma(1 - t_k)} + \sum_{j=0}^{k-1} \abs{\gamma(1 - t_{j+1}) - \gamma(1 - t_j)}.
 \end{gather*}
 As
  \(\gamma(t) = - \gamma(t + \frac12)\)
 we can shorten this expression to twice the first half.
 Then we substitute in from corollary~\ref{cor:xival} to find that the total variation with respect to \(\m{P}\) is:
 \[
  4\abs{\xi_k} + 4 \sum_{j=0}^{k-1} \abs{\xi_j} = 4 \sum_{j=0}^k \abs{\xi_j} = 4 \norm[\gamma]_B.
 \]
 As this is independent of \(\m{P}\) we see that \(\gamma\) is of bounded variation with total variation \(4 \norm[\gamma]_B\).
\end{proof}

\begin{corollary}
 The completion of \(E_B\) is the space of \(S^1\)\enhyp{}odd loops of bounded variation, whence the completion of \(E_A\) is the space of differentiable \(S^1\)\enhyp{}odd loops with derivative of bounded variation.
 \noproof
\end{corollary}

Had we assumed that the map
 \(S^1 \to (L_{\psb}\R, \m{T})\),
 \(t \mapsto R_t \alpha_0\)
was Lipschitz we would have obtained the space \(L^{1,1}(S^1, \R)\) of differentiable \(S^1\)\enhyp{}odd loops with Lebesgue integrable derivative.

\begin{proposition}
 Let \(\m{T}\) be a locally convex topological vector space topology on \(L_{\psb}\R\) satisfying the following conditions:
 \begin{enumerate}
 \item
  The maps
   \(L_{F,b}\R \to L_{\psb}\R\)
  and \(L_{\psb}\R \to L^0\R\)
  are continuous with respect to \(\m{T}\).

 \item
  The completion of \((L_{\psb}\R, \m{T})\) injects into \(L^0\R\).

 \item
  The circle action is separately continuous.
 \end{enumerate}
 Then the restriction of \(\m{T}\) to the subspace of \(S^1\)\enhyp{}odd loops is at least as coarse as the topology given by the norm
  \(\int_{S^1} \abs{\gamma^{(2)}(s)} d s + \sum_{t \in S^1} \abs{\gamma^{(2)}_+(t) - \gamma^{(2)}_-(t)}\).
\end{proposition}

\begin{proof}
 The image in \(L^0\R\) of the completion of \(E_A\) contains the subspace of \(S^1\)\enhyp{}odd piecewise\hyp{}smooth loops hence by the injectivity assumption the completion of \(E_A\) in the completion of \((L_{\psb}\R, \m{T})\) must contain the subspace of \(S^1\)\enhyp{}odd piecewise\hyp{}smooth loops.
 Thus \(E_A\) is dense in the subspace of \(S^1\)\enhyp{}odd piecewise\hyp{}smooth loops and so the topology on this subspace is completely determined by its restriction to \(E_A\).
 This topology is normable with norm given by the total variation of the first derivative.
 For a smooth loop this is
  \(\int_{S^1} \abs{\gamma^{(2)}(s)} d s\)
 whilst for a piecewise\hyp{}smooth loop we merely need to add in the absolute values of the breaks in \(\gamma^{(2)}\).
\end{proof}


\begin{thebibliography}{Omo97}

\bibitem[Che73]{kc}
Kuo-tsai Chen.
\newblock Iterated integrals of differential forms and loop space homology.
\newblock {\em Ann. of Math. (2)}, 97:217--246, 1973.

\bibitem[Che77]{kc3}
Kuo~Tsai Chen.
\newblock Iterated path integrals.
\newblock {\em Bull. Amer. Math. Soc.}, 83(5):831--879, 1977.

\bibitem[Jar81]{hj}
Hans Jarchow.
\newblock {\em Locally convex spaces}.
\newblock B. G. Teubner, Stuttgart, 1981.
\newblock Mathematische Leitf{\"a}den. [Mathematical Textbooks].

\bibitem[Kli95]{wk}
Wilhelm P.~A. Klingenberg.
\newblock {\em Riemannian geometry}, volume~1 of {\em de Gruyter Studies in Mathematics}.
\newblock Walter de Gruyter \& Co., Berlin, second edition, 1995.

\bibitem[KM97]{akpm}
Andreas Kriegl and Peter~W. Michor.
\newblock {\em The convenient setting of global analysis}, volume~53 of {\em Mathematical Surveys and Monographs}.
\newblock American Mathematical Society, Providence, RI, 1997.

\bibitem[Lan85]{sl}
Serge Lang.
\newblock {\em Differential manifolds}.
\newblock Springer-Verlag, New York, second edition, 1985.

\bibitem[Mic80]{pm3}
Peter~W. Michor.
\newblock {\em Manifolds of differentiable mappings}, volume~3 of {\em Shiva Mathematics Series}.
\newblock Shiva Publishing Ltd., Nantwich, 1980.

\bibitem[Mil84]{jm}
J.~Milnor.
\newblock Remarks on infinite-dimensional {L}ie groups.
\newblock In {\em Relativity, groups and topology, II (Les Houches, 1983)}, pages 1007--1057. North-Holland, Amsterdam, 1984.

\bibitem[Omo97]{ho}
Hideki Omori.
\newblock {\em Infinite-dimensional {L}ie groups}, volume 158 of {\em  Translations of Mathematical Monographs}.
\newblock American Mathematical Society, Providence, RI, 1997.
\newblock Translated from the 1979 Japanese original and revised by the author.

\bibitem[Pie72]{ap}
Albrecht Pietsch.
\newblock {\em Nuclear locally convex spaces}.
\newblock Springer-Verlag, New York, 1972.
\newblock Translated from the second German edition by William H. Ruckle, Ergebnisse der Mathematik und ihrer Grenzgebiete, Band 66.

\bibitem[PS86]{apgs}
Andrew Pressley and Graeme Segal.
\newblock {\em Loop groups}.
\newblock Oxford Mathematical Monographs. The Clarendon Press Oxford University Press, New York, 1986.
\newblock Oxford Science Publications.

\bibitem[Sch71]{hs}
Helmut~H. Schaefer.
\newblock {\em Topological vector spaces}.
\newblock Springer-Verlag, New York, 1971.
\newblock Third printing corrected, Graduate Texts in Mathematics, Vol. 3.

\bibitem[See64]{rs}
R.~T. Seeley.
\newblock Extension of {$C\sp{\infty }$} functions defined in a half space.
\newblock {\em Proc. Amer. Math. Soc.}, 15:625--626, 1964.

\bibitem[Staa]{math.DG/0612096}
Andrew Stacey.
\newblock {Constructing Smooth Loop Spaces}, arXiv:math.DG/0612096.

\bibitem[Stab]{math.DG/0505077}
Andrew Stacey.
\newblock {The Geometry of the Loop Space and a Construction of a Dirac Operator}, arXiv:math.DG/0505077.

\bibitem[Wel73]{jw}
John~C. Wells.
\newblock Differentiable functions on {B}anach spaces with {L}ipschitz derivatives.
\newblock {\em J. Differential Geometry}, 8:135--152, 1973.

\end{thebibliography}
\end{document}